\documentclass{elsarticle}
\usepackage{amssymb,wrapfig,amsmath, amsthm, amscd,ifthen}
\usepackage[cp1251]{inputenc}
\usepackage{color} %Color definition
\usepackage{hyperref} %Internal and external links
\usepackage{graphicx,caption,subcaption,mathrsfs,appendix}
\newtheorem{theorem}{Theorem}
\newtheorem{lemma}[theorem]{Lemma}

\newcommand{\E}{\mathbb E}
\newcommand{\N}{\mathbb N}

\begin{document}
\title{$\gamma$-variable first-order logic of uniform attachment random graphs\tnoteref{t1}}
\tnotetext[t1]{Maksim Zhukovskii is supported by the Ministry of Science and Higher Education of the Russian Federation in the framework of MegaGrant no 075-15-2019-1926. The part of the study made by Y.A. Malyshkin was funded by RFBR, project number 19-31-60021.}
\author[1]{Y.A. Malyshkin}
\ead{yury.malyshkin@mail.ru}
\author[2]{M.E. Zhukovskii}
\ead{zhukmax@gmail.com}

%\author{Yury Malyshkin, Maksim Zhukovskii}
%\address{Moscow Institute of Physics and Technology}
%\keywords{uniform attachment, random graphs, 0-1 law, first-order logic}
%\date{}

\address[1]{Moscow Institute of Physics and Technology//Tver State University}
\address[2]{Moscow Institute of Physics and Technology}

\begin{abstract}
We study logical limit laws for uniform attachment random graphs. In this random graph model, vertices and edges are introduced recursively: at time $n+1$, the vertex $n+1$ is introduced together with $m$ edges joining the new vertex with $m$ different vertices chosen uniformly at random from $1,\ldots,n$. We prove that this random graph obeys convergence law for first-order sentences with at most $m-2$ variables.
\end{abstract}


\maketitle

\section{Introduction}

The well-known first-order (FO) zero-one law for finite models~\cite{Fagin,Glebskii} states that, for every FO sentence $\phi$, a $\sigma$-structure (a vocabulary $\sigma$ is given) with the universe $[n]:=\{1,\ldots,n\}$ chosen uniformly at random satisfies $\phi$ with an asymptotical probability either 0 or 1, $n\to\infty$. For graphs, this law can be reformulated in the following way. For every FO sentence $\varphi$ over graphs (where $\sigma$ consists of the relations of equality $=$ and adjacency $\sim$ of vertices in the graph), the  probability $\Pr(G(n,1/2)\models\varphi)$ that the binomial random graph $G(n,1/2)$~\cite{Janson,Survey} satisfies $\varphi$ converges to either 0 or 1, $n\to\infty$. For other constant edge probability functions $p$, it is known that $G(n,p)$ obeys FO zero-one law as well~\cite{Spencer_Ehren}. However, for $p=n^{-\alpha}$, the situation changes~\cite{Shelah}: zero-one law holds if and only if $\alpha$ is either irrational or bigger than 1 and does not equal $1+1/\ell$ for $\ell\in\mathbb{N}$. Moreover, when $\alpha\in(0,1)$ is rational, even the FO convergence law fails: there are FO sentences $\varphi$ such that  $\Pr(G(n,n^{-\alpha})\models\varphi)$ do not converge as $n\to\infty$, \cite{Shelah}. In this paper, we study FO convergence laws on uniform attachment random graph models.\\

Let us recall that FO sentences about graphs comprise the following symbols: variables $x,y,x_1,\ldots$ (which represent vertices), logical connectives $\wedge,\vee,\neg,\Rightarrow,$ $\Leftrightarrow$, two relational symbols (between variables) $\sim$ (adjacency) and $=$ (equality), brackets and quantifiers $\exists,\forall$  (see the formal definition in, e.g.,~\cite{Libkin,Survey,Strange}). For example, the sentence
$$
\forall x \forall y\quad [\neg(x=y)\wedge\neg(x\sim y)]\Rightarrow[\exists z\,\,(x\sim y)\wedge(z\sim y)] 
$$
expresses the property of having diameter at most 2. Following standard notations of model theory, we write $G\models\varphi$ when FO sentence $\varphi$ is true on graph $G$.

Let $\mathcal{G}_n$ be a random graph on the vertex set $[n]$ with a random set of edges. For a graph property $Q$, we say that it {\it a.a.s.} (asymptotically almost surely) holds for  $\mathcal{G}_n$, if ${\sf P}(\mathcal{G}_n\in Q)\to 1$ as $n\to\infty$. The sequence of random graphs $\{\mathcal{G}_n\}_{n\in\mathbb{N}}$ {\it obeys FO zero-one law}, if, for every FO sentence $\varphi$, $\lim_{n\to\infty}\Pr(\mathcal{G}_n\models\varphi)\in\{0,1\}$ (in other words, a.a.s. $G_n\models\varphi$). It {\it obeys FO convergence law}, if, for every FO sentence $\varphi$, $\Pr(\mathcal{G}_n\models\varphi)$ converges as $n\to\infty$.

Many random graph models are well-studied in the context of logical limit laws. FO zero-one laws and convergence laws were established for the binomial random graph (\cite{Spencer_Ehren,Shelah}), random regular graphs (\cite{Haber}), random geometric graphs (\cite{geo}), uniform random trees (\cite{McColm}), and many others (see, e.g., \cite{Muller,Strange,Zhuk_Svesh,Winkler}).
However, for recursive random graph models, the only attempt to prove logical laws was made by R.D. Kleinberg and J.M. Kleinberg~\cite{Kleinberg}. In that paper, it was noticed that  the preferential attachment random graph with parameter $m$ (the number of edges that appear at every step) does not obey FO zero-one law when $m\geq 3$. In our recent paper~\cite{MZ20}, we proved that, if $m=1$, then both the preferential attachment random graph and the uniform attachment random graph obey FO zero-one law. 

Let us recall that {\it the uniform attachment random graph}~\cite{recursive,dags,Magner,Mahmoud} is generated in the following way. We initially begin with a complete graph on $m$ vertices (so $G_{m,m}\cong K_m$). Graph $G_{n+1,m}$ is built from $G_{n,m}$ by adding the new vertex $n+1$ and drawing $m$ edges from it to different vertices of $G_{n,m}$ chosen uniformly at random. In particular, it means that, for a given vertex $v\leq n$, the probability of adding an edge to it at step $n+1$ is exactly $\frac{m}{n}$. In~\cite{MZ20}, we showed that, for $m\geq 2$, $G_{n,m}$ does not obey FO zero-one law. However, the question about validity of the FO convergence law is still open. In this paper, we prove that the FO convergence law holds for sentences with at most $m-2$ variables.\\

We say that $\{\mathcal{G}_n\}_{n\in\mathbb{N}}$ obeys $\mathrm{FO}^{\gamma}$ {\it convergence law} if, for every FO sentence $\varphi$ with at most $\gamma$ different variables, $\Pr(\mathcal{G}_n\models\varphi)$ converges as $n\to\infty$. In this paper, we prove the following.

\begin{theorem}
\label{m-law}
$\{G_{n,m}\}_{n\in\mathbb{N}}$ obeys $\mathrm{FO}^{m-2}$ convergence law.
\end{theorem}

{\it Remark}. Let us also recall that the {\it quantifier depth} of a FO sentence $\varphi$ is, roughly speaking, the maximum length of a sequence of nested quantifiers in $\varphi$ (see the formal definition in \cite[Definition 3.8]{Libkin}). It is straightforward that any FO sentence with quantifier depth $q$ has a tautologically equivalent FO sentence with at most $q$ variables. Therefore, Theorem~\ref{m-law} implies the validity of the convergence law for FO sentences with quantifier depth at most $m-2$. However, it is not hard to see that the same proof works even for FO sentences with quantifier depth at most $m-1$.\\

Let us notice that the study of the fragment $\mathrm{FO}^{\gamma}$ of the FO logic in the context of limit laws is in full accordance with the finite model theory since proving or disproving logical limit laws leads to better understanding of the hierarchy of these fragments which, in turn, is strongly related to estimation of time complexity of decision problems formulated in the respective logics~(see \cite[Chapter 6]{Libkin}).\\

The paper is organized as follows. In Section~\ref{proof}, we prove Theorem~\ref{m-law}. The proof is based on two auxiliary statements. The first one describes some local properties of the random graph and is proven in Section~\ref{sec:prob}. The second one claims that, in the $\gamma$-pebble Ehrenfeucht-Fra\"{\i}ss\'{e} game on two graphs $G_1\subset G_2$ (i.e., $G_1$ is a subgraph of $G_2$) with the local properties described in the first statement, Duplicator wins. The proof of the second statement is given in Section~\ref{proof_claim_m-game}. In Section~\ref{conj}, we conjecture that $\{G_{n,m}\}_{n\in\mathbb{N}}$ obeys FO convergence law and describe a possible approach to prove that.

\section{Proof of Theorem~\ref{m-law}}
\label{proof}
One of the main tools to prove FO logical limit laws is the Ehrenfeucht-Fra\"{\i}ss\'{e} pebble game (see, e.g., \cite[Chapter 11.2]{Libkin}). Let us recall the rules of the game.

The {\it $\gamma$-pebble game} is played on two graphs $G$ and $H$ with $\gamma$ pebbles assigned to each of them (say, $g_1,\ldots,g_{\gamma}$ and $h_1,\ldots,h_{\gamma}$). There are two players, {\it Spoiler} and {\it Duplicator}. In each round, Spoiler moves a pebble to a vertex either in $G$ or in $H$; then Duplicator must move the pebble with the same subscript to a vertex in the other graph (hereinafter, we refer to a vertex containing a pebble as {\it pebbled vertex}). Let $x^i_1,\ldots,x^i_{\gamma}\in V(G)$ and $y^i_1,\ldots,y^i_{\gamma}\in V(H)$ denote the pebbled vertices in the $i$-th round. Given $R\in\N$, if for every $i\leq R$, and for all $j,k\in\{1,...,\gamma\}$ we have $x_{j}^{i}=x_{k}^{i}$ iff $y_j^i=y_k^i$ and $x_{j}^i\sim x_{k}^i$ iff $y_{j}^i\sim y_{k}^i$, i.e. $G|_{\{x^i_1,\ldots,x^i_{\gamma}\}}$ is isomorphic to $H|_{\{y^i_1,\ldots,y^i_{\gamma}\}}$ (hereinafter, we denote by $G|_A$ the subgraph of $G$ induced on the set of vertices $A\subset V(G)$), then Duplicator wins the $\gamma$-pebble game of $R$ rounds.
%Duplicator wins the game in $R$ rounds if, for every $i\leq R$, the componentwise correspondence between the ordered $\gamma$-tuples $x^i_1,\ldots,x^i_{\gamma}$ and $y^i_1,\ldots,y^i_{\gamma}$ is an isomorphism from $G|_{\{x^i_1,\ldots,x^i_{\gamma}\}}$ to $H|_{\{y^i_1,\ldots,y^i_{\gamma}\}}$ (hereinafter, we denote by $G|_A$ the subgraph of $G$ induced on the set of vertices $A\subset V(G)$). 
Otherwise the winner is Spoiler. If Duplicator can play in a clever way that guarantees a winning position in the last round, no matter how Spoiler plays, then we say that Duplicator has {\it a winning strategy}.

The well-known relation between FO logics with a finite number of variables and pebble games is as follow (see, e.g. \cite[Theorem 11.5]{Libkin}).

\begin{theorem}
Duplicator has a winning strategy for the $\gamma$-pebble game on $G$ and $H$ in $R$ rounds if and only if, for every FO sentence $\varphi$ with at most $\gamma$ variables and quantifier depth at most $R$, either $\varphi$ is true on both $G$ and $H$ or it is false on both graphs.
\label{Ehren}
\end{theorem}

Note that the existence of the winning strategy for Duplicator always follows from a `local structure' of graphs. In other words, for every $R\in\mathbb{N}$, there exists $a(R)$ such that, in order to verify a given property expressed by a FO sentence of quantifier depth at most $R$, it is sufficient to know $a(R)$-neighborhoods of all vertices.
So, we start from describing `sufficient' local properties of the random graph sequence $\{G_{n,m}\}_{n\in\mathbb{N}}$.
%Section~\label{trees_cycles} is devoted to definitions of constructions that `represent' a local structure of the recursive random graphs. In Section~\label{sec:prob} we prove local properties certain local properties of the graph. Then in Section~\ref{uniform_proof} we would use them to construct a winning strategy in Ehrenfeucht game and prove Theorem~\ref{m-law}.

\begin{lemma} 
\label{lem:prob_main}

Let $a\in\mathbb{N}$. 

\begin{enumerate}

\item For every   $\varepsilon>0$, there exist $n_0=n_0(\varepsilon)\in\mathbb{N}$, $N_0=N_0(\varepsilon)\in\mathbb{N}$ such that, for every $n\geq n_0$, with probability at least $1-\varepsilon$, 

\vspace{0.1cm}

--- for every cycle of $G_{n,m}$ with at most $a$ vertices, it either has all vertices inside $[N_0]$, and any path connecting $[n_0]$ with this cycle and having length at most $a$ has all vertices inside $[N_0]$, or is at distance at least $a$ from $[n_0]$; 

\vspace{0.1cm}

--- every path with at most $a$ vertices joining two vertices of $[n_0]$ has all vertices inside $[N_0]$;

\vspace{0.1cm}

--- any two cycles with all vertices in $[n]\setminus[n_0]$ and having at most $a$ vertices are at distance at least $a$ from each other.

\vspace{0.1cm}

%--- for every simple path of length at most $a$ between any two vertices of $[n_0]$, all its vertices are inside $[N_0]$,

%--- $\varepsilon_r\downarrow 0$ as $r\to\infty$.

\item For every $N_0\in\mathbb{N}$ and $K\in\mathbb{N}$, a.a.s., for every $b\leq a$, there are at least $K$ distinct copies of $C_b$ (as usual, $C_b$ stands for a cycle of length $b$) in $G_{n,m}$ with all vertices in $[n]\setminus[N_0]$.%, and any two cycles with all vertices in $[n]\setminus[N_0]$ and having at most $a$ vertices are at distance at least $a$ from each other.

%\item A.a.s., there is a vertex $v$ such that the distance between $v$ and $[n_0]$ is bigger than $a$ and the distance between $n$ and the closest cycle is bigger than $a$.

\item For every $N_0\in\mathbb{N}$ and  $K\in\mathbb{N}$, a.a.s. every vertex of $[N_0]$ has degree at least $K$ in $G_{n,m}$.

\end{enumerate}

\end{lemma}

Lemma~\ref{lem:prob_main} is proven in Section~\ref{sec:prob}.\\

Assume that a FO sentence $\varphi$ with at most $m-2$ variables and quantifier depth $R\geq m-2$ has no limit probability. Then there exist $p_1,p_2\in[0,1]$ with $p_1>p_2$ and increasing sequences $\{n^1_i\}_{i\in\mathbb{N}}$, $\{n^2_i\}_{i\in\mathbb{N}}$ of positive integers such that $\Pr\left[G_{n^1_i,m}\models\varphi\right]\geq p_1$, $\Pr\left[G_{n^2_i,m}\models\varphi\right]\leq p_2$ for every $i\in\mathbb{N}$. 

Fix $\varepsilon=\frac{p_1-p_2}{4}$. Set $a=3^R$. %Let $r$ be such that $\varepsilon_r<\varepsilon$ ($\varepsilon_r$ is defined in Lemma~\ref{lem:prob_main}).
Set $n_0=n_0(\varepsilon)$, $N_0=N_0(\varepsilon)$ (their existence is stated in Lemma~\ref{lem:prob_main}). Define the following properties of graphs on the vertex set $[n]$.
\begin{itemize}
\item[${\sf Q1}$] The following conditions hold.
\begin{itemize}
\item For every cycle with at most $a$ vertices,

\begin{itemize}

\item either it has all vertices inside $[N_0]$, and any path connecting $[n_0]$ with this cycle and heaving length at most $a$ has all vertices inside $[N_0]$, 
 
\item or it is at distance at least $a$ from $[n_0]$; 

\end{itemize}

\item every path with at most $a$ vertices joining two vertices of $[n_0]$ has all vertices inside $[N_0]$; 

\item any two cycles with all vertices in $[n]\setminus[n_0]$ and having at most $a$ vertices are at distance at least $a$ from each other.
\end{itemize}
%, and every simple path of length at most $a$ joining two vertices of $[n_0]$ has all vertices inside $[N_0]$,
\item[${\sf Q2}$] For every $b\leq a$, there exist at least $m$ distinct copies of $C_b$ with all vertices in $[n]\setminus[N_0]$.% and, any two cycles in $[n]\setminus[N_0]$ with at most $a$ vertices are at distance at least $a$ from each other,
\item[${\sf Q3}$] Every vertex of $[N_0]$ has degree at least $N_0+m$.
\end{itemize}

By Lemma~\ref{lem:prob_main}, a.a.s. $G_{n,m}$ has properties ${\sf Q2}$, ${\sf Q3}$, and the property ${\sf Q1}$ holds with probability at least $1-\varepsilon$. Theorem~\ref{m-law} follows from the lemma below. The lemma itslef is proven in Section~\ref{proof_claim_m-game}.

\begin{lemma}
Let $H_1,H_2$ be graphs on vertex sets $[n_1]$ and $[n_2]$ respectively with minimum degrees at least $m$. Let $H_1|_{[N_0]}= H_2|_{[N_0]}$ and both $H_1$ and $H_2$ have properties ${\sf Q1}$, ${\sf Q2}$, ${\sf Q3}$. Then Duplicator wins the $(m-2)$-pebble game on $H_1$ and $H_2$ in $R$ rounds.
\label{claim_m-game}
\end{lemma}

Indeed, for $i$ satisfying $\min\{n_i^1,n_i^2\}>N_0$,
$$
 \Pr\left(G_{n^1_i,m}\models\varphi,G_{n^2_i,m}\not\models\varphi\right)\geq
 \Pr\left(G_{n^1_i,m}\models\varphi\right)-\Pr\left(G_{n^2_i,m}\models\varphi\right)\geq p_1-p_2.
$$
By Theorem~\ref{Ehren}, for the event 
$$\mathcal{A}_{n,m}:=\left\{\text{Spoiler wins } R\text{-round }(m-2)\text{-pebble game on } G_{n^1_i,m},G_{n^2_i,m}\right\}$$
we get
$$
 \Pr\left({A}_{n,m}\wedge\bigwedge_{\ell=1}^2\bigwedge_{j=1}^3\left\{G_{n^{\ell}_i,m}\in{\sf Qj}\right\}\right)
$$
$$
 \geq\Pr\left(\left\{G_{n^1_i,m}\models\varphi,G_{n^2_i,m}\not\models\varphi\right\}\wedge\bigwedge_{\ell=1}^2\bigwedge_{j=1}^3\left\{G_{n^{\ell}_i,m}\in{\sf Qj}\right\}\right)
$$
$$
 \geq\Pr\left(G_{n^1_i,m}\models\varphi,G_{n^2_i,m}\not\models\varphi\right)-\sum_{\ell\in\{1,2\},j\in\{1,2,3\}}\Pr\left(G_{n^{\ell}_i,m}\notin{\sf Qj}\right)
$$
$$ 
\geq p_1-p_2-2\varepsilon-o(1)=\frac{p_1-p_2}{2}-o(1)
$$
which is bounded away from 0, and that contradicts Lemma~\ref{claim_m-game}.

\section{Proof of Lemma~\ref{lem:prob_main}}
\label{sec:prob}

To prove Lemma~\ref{lem:prob_main}, we first need the following standard facts (see, e.g.,~\cite{dags}) about the maximum degree and the number of cycles of a giving length in uniform attachment random graphs. We give here short proofs for the sake of convenience.\\

Let $\deg_n j$ be the degree of vertex $j$ in $G_{n,m}$ and $\Delta_n=\max_{j\in[n]}\deg_{n}j$ be the maximum degree of $G_{n,m}$. 
\begin{lemma}
\label{lem:max_degree}
For any $\epsilon>0$ there is a constant $c_{\epsilon}$, such that 
$$
\Pr(\forall n\quad \Delta_n>c_{\epsilon} (\ln n)^2)<\epsilon .$$
\end{lemma}
\begin{proof}
Clearly, for every $j<n$ the probability that $j\sim n$ in $G_{n,m}$ is exactly $\frac{m}{n-1}$. Since drawing edges in different time-steps
of generating the graph are independent from each other, we get that, by Markov inequality, for any $C>0$,
$$
\Pr\left(\deg_n j\geq C\left(\ln n\right)^2\right)
=\Pr\left(\exp\left[\deg_n j\right]\geq \exp\left[C\left(\ln n\right)^2\right]\right)
$$
$$
\leq e^{m-C\left(\ln n\right)^2}\prod_{i=j}^{n-1}\left(1+\frac{m(e-1)}{i}\right)
\leq e^{-C\left(\ln n\right)^2+m(e-1)\ln n+O(1)}.
$$
Therefore, 
$$
\Pr\left(\exists n\quad\Delta_n\geq C\left(\ln n\right)^2\right)\leq \sum_{n=m+1}^{\infty} ne^{-C\left(\ln n\right)^2+m(e-1)\ln n+O(1)}.
$$
Since the right side of this inequality converges, and approaches 0 as $C\to\infty$, Lemma~\ref{lem:max_degree} follows. 
\end{proof}

Notice that, in the proof of Lemma~\ref{lem:max_degree} we show that, uniformly in $n\in\mathbb{N}$,
\begin{equation}
\Pr(\Delta_n\geq C\left(\ln n\right)^2)\leq e^{-C\left(\ln n\right)^2(1+o(1))}.
\label{prob_Delta}
\end{equation}

Let $\mathcal{C}_k(n)$ be the number of cycles of length $k$ in $G_{n,m}$.
\begin{lemma}
\label{lem:number_of_cycles}
For any $\epsilon>0$ there is a constant $c'_{\epsilon}$, such that 
$$\Pr\left(\forall n\quad \mathcal{C}_k(n)>c'_{\epsilon} (\ln n)^{2(k-2)^2+2}\right)<\epsilon .$$
\end{lemma}
\begin{proof}
Let $c_{\varepsilon/2}$ be the constant from Lemma~\ref{lem:max_degree}.
The event 
$$
\mathcal{D}_n:=\left\{\Delta_n\leq c_{\varepsilon/2}\left(\ln n\right)^2\right\}
$$ 
implies that every vertex in $G_{n,m}$ has a $(k-2)$-neighbourhood of size at most $F_n:=(c_{\varepsilon/2}\ln^2 n)^{k-2}$. Then, the probability of forming a cycle of length $k$ at time $n+1$ (or, in other words, to join vertex $n+1$ to two vertices such that the second one is in the $(k-2)$-neighborhood of the first one) conditioned on $\mathcal{D}_n$ does not exceed $m(m-1)F_n/n$. Let $\xi_n$ be a Bernoulli random variable that equals 1 if the vertex $n$ belongs to a $k$-cycle in $G_{n,m}$. From above, there exist independent Bernoulli random variables $\tilde\xi_n$, $n\geq m+1$, with success probability $\frac{m(m-1)F_n}{n}+1-\Pr(\mathcal{D}_n)$ such that $\xi_n\leq\tilde\xi_n$ for every $n$. We get that, for every $C>0$,
$$
\Pr\left(\xi_{m+1}+\ldots+\xi_n>C\left(\ln n\right)^{2k-2}\right)\leq
\Pr\left(\tilde\xi_{m+1}+\ldots+\tilde\xi_n>C\left(\ln n\right)^{2k-2}\right)
$$
$$
=\Pr\left(\exp\left[\tilde\xi_{m+1}+\ldots+\tilde\xi_n\right]>\exp\left[C\left(\ln n\right)^{2k-2}\right]\right)
$$
$$
\leq e^{-C(\ln n)^{2k-2}}\prod_{i=m}^{n-1}\left(1+(e-1)\frac{m(m-1)F_i}{i}+1-\Pr(\mathcal{D}_n)\right).
$$
The bound~(\ref{prob_Delta}) implies that
$$
\Pr\left(\xi_{m+1}+\ldots+\xi_n>C\left(\ln n\right)^{2k-2}\right)\leq e^{-C(\ln n)^{2k-2}+O\left([\ln n]^{2k-3}\right)}.
$$
Since $\sum_ne^{-C(\ln n)^{2k-2}(1+o(1))}$ converges and approaches 0 as $C\to\infty$, there exists $ c'_{\varepsilon}$ such that 
$$
\Pr\left(\exists n\quad \xi_{m+1}+\ldots+\xi_n>c'_{\varepsilon}\left(\ln n\right)^{2k-2}\right)<\varepsilon/2.
$$ 
Finally, conditioned on $\mathcal{D}_n$, the vertex $n+1$ gives at most ${m\choose 2}F_n^{k-3}$ new $k$-cycles. Since $\Pr\left(\bigcap_{n\in\mathbb{N}}\overline{\mathcal{D}_n}\right)<\varepsilon/2$, (as usual, $\overline{\mathcal{D}_n}$ is the complement event of $\mathcal{D}_n$) Lemma~\ref{lem:number_of_cycles} follows. \end{proof}

Now let us prove Lemma~\ref{lem:prob_main}.
\begin{proof}
Let us prove the first part of Lemma~\ref{lem:prob_main}. 
%{\bf [Note that there are two ways to form two cycles close to each other at step $n+1$. The one is to form a new cycle near already existing cycle, which would require to draw at least two edges to its neighborhood. The other is to form two cycles at the same time, this way they would share a vertex $n+1$. This could be doneby forming two cycles that either  share an edge adjacent to $n+1$ (to do so one needs to draw an edge to a vertex and then draw two edges to $a-2$-neighborhood of it) or does not share adjacent to $n+1$ (to do so one needs to draw an edge to some vertex and its $a-2$-neighborhood and then draw an edge to another vertex and its $a-2$-neighborhood).]}
Due to Lemma~\ref{lem:max_degree} and Lemma~\ref{lem:number_of_cycles}, for any $\varepsilon>0$, there is a constant $c_{\varepsilon}$ such that, with probability at least $1-\varepsilon/3$, for every $n$, 
\begin{equation}
\Delta_n\leq c_{\varepsilon}\left(\ln n\right)^2
\label{cond1}
\end{equation}
 and, for any $a\in\N$, the union of $3a$-neighborhoods of all cycles of length at most $a$ (denote this union by $\mathcal{U}_a(n)$) contains at most $F_n:=a^2c'_{\epsilon}(c_{\varepsilon})^{3a}(\ln n)^{2a^2-2a+10}$ vertices (the power of $\ln n$ is obtained by combining Lemma~\ref{lem:number_of_cycles} with (\ref{cond1}): $2(a-2)^2+2+6a=2a^2-2a+10$). Under the condition that 
\begin{equation} 
\left|\mathcal{U}_a(n)\right|\leq F_n,
\label{cond2}
\end{equation} 
the probability that $n+1$ is adjacent to at least two vertices of $\mathcal{U}_a(n)$ in $G_{n+1,m}$ does not exceed $m^2F^2_n/n^2$. At the same time, under Condition (\ref{cond1}), the $(a-2)$-neighborhood of a vertex contains at most $(c_{\varepsilon}\left(\ln n\right)^2)^{a-2}$ vertices. It implies that, in $G_{n+1,m}$, the vertex $n+1$ belongs to two cycles of length at most $a$ that share at most one edge adjacent to $n+1$, with probability not exceeding $2m^4a^2c_{\varepsilon}^{2a-4}(\ln n)^{4a-8}/n^{2}$. Indeed, if such pair of cycles has one common edge adjacent to $n+1$ (say, the edge $\{n+1,v\}$), then, after the choice of the vertex $v$ (note that one of $m$ edges drawn from $n+1$ plays the role of $\{n+1,v\}$), each of the other two neigbors of $n+1$ in these two cycles can be chosen in at most $ma(c_{\varepsilon}\left(\ln n\right)^2)^{a-2}$ ways. On the other hand, if a pair of cycles does not have a common edge adjacent to $n+1$, then, after the choice of neighbors $v_1,v_2$ of $n+1$ within these two cycles, the other pair of vertices can be chosen in at most $ \left[ma(c_{\varepsilon}\left(\ln n\right)^2)^{a-2}\right]^2$ ways.

% (with different constant $c_{\varepsilon}$). 

Note that two cycles of length at most $a$ and at distance at most $a$ from each other can be formed at step $n+1$ either by joining $n+1$ with two vertices of $\mathcal{U}_a(n)$ or by drawing two cycles that share the vertex $n+1$ and at most 1 edge adjacent to $n+1$. Since both $\sum_n F^2_n/n^2$ and $\sum_{n}(\ln n)^{4a-8}/n^{2}$ converge, there exists $n_0$ such that, under the condition that (\ref{cond1}) and  (\ref{cond2}) hold for all $n$, the probability that, for any $n>n_0$, any two cycles with all vertices in $[n]\setminus[n_0]$ and having at most $a$ vertices are at distance at least $a$ from each other is at least $1-\varepsilon/3$. 

Now, let $\mathcal{U}_{[n_0],a}(n)$ be the $a$-neighborhood of $[n_0]$ in $G_{n,m}$. (\ref{cond1}) implies that it contains at most $an_0c_{\varepsilon}^{a}(\ln n)^{2a}$ vertices. Hence, under Conditions (\ref{cond1}) and (\ref{cond2}),  the probability that $n+1$ is adjacent to at least 2 vertices of $\mathcal{U}_{[n_0],a}(n)\bigcup \mathcal{U}_{a}(n)$ in $G_{n+1,m}$ does not exceed $m^2\left[an_0c_{\varepsilon}^{a}(\ln n)^{2a}+a^2c'_{\epsilon}(c_{\varepsilon})^{3a}(\ln n)^{2a^2-2a+10}\right]^2/n^2$. Therefore, there exists $N_0>n_0$ such that, with probability $1-\varepsilon/3$, for all $n\geq N_0$ the vertex $n+1$ is adjacent to at most one vertex of $\mathcal{U}_{[n_0],a}(n)\bigcup \mathcal{U}_{a}(n)$ in $G_{n+1,m}$. Part 1 of Lemma~\ref{lem:prob_main} follows.\\

Let us switch to the second part. From \eqref{prob_Delta}, for large enough $n$ (say, $n\geq N$), with probability at least $1-e^{-\left(\ln n\right)^2(1+o(1))}$, there are at least $n/2$ vertices at distance at least $a-1$ from $[N_0]$ (since its $(a-1)$-neighborhood has size $O\left((\ln n)^{2a-2}\right)$ which is less then $n/2$ for large $n$). Let $b\leq a$. The latter event implies that there are at least $n/4$ pairs of vertices joined by a simple path of length $b-2$ having all vertices outside $[N_0]$ (each of these at least $n/4$ simple paths can be obtained in the following way: connect a vertex which is at distance $a-1$ from $[N_0]$ with $[N_0]$ by a shortest path, and take the initial part of length $b-2$ of this path). To create a cycle of length $b$ at step $n+1$, we have to connect new vertex $n+1$ with two vertices at distance $b-2$ from each other. Therefore, with probability at least $1-e^{-\left(\ln n\right)^2(1+o(1))}$, there are at least $n/4$ possibilities out of ${n\choose 2}$ for first two edges drawn from $n+1$ to create a desired cycle. Let the Bernoulli random variable $\xi_{n+1}$ equal 1 if and only if $n+1$ belongs to a $b$-cycle in $G_{n+1,m}$ having all vertices outside $[N_0]$. Clearly, there exist independent Bernoulli random variables $\tilde\xi_{N+1},\tilde\xi_{N+2},\ldots$ such that, for every $j\in\{N+1,N+2,\ldots\}$, $\xi_j\geq\tilde\xi_j$ and $\Pr(\tilde\xi_j=1)=\frac{1}{2j}-e^{-\left(\ln j\right)^2(1+o(1))}$ ({\it uniformly in $j$}, i.e. the $o(1)$ can be bounded by a sequence approaching 0 and not depending on $j$).  Therefore, by Markov's inequality, for $n>N$,
\begin{align*}
 \Pr(\xi_{N+1}+\ldots+\xi_n<K) & \leq\Pr(\tilde\xi_{N+1}+\ldots+\tilde\xi_n<K)\\
 &=\Pr(e^{-(\tilde\xi_{N+1}+\ldots+\tilde\xi_n)}>e^{-K})\leq e^{K}\prod_{j=N+1}^n\E e^{-\tilde\xi_j}\\
&\leq e^K\prod_{j=N+1}^n\left(1-(1-1/e)\left[\frac{1}{2j}-e^{-\left(\ln j\right)^2(1+o(1))}\right]\right)\\
&=e^{K+\sum_{j=N+1}^{n}\ln\left(\left(1-(1-1/e)\left[\frac{1}{2j}-e^{-\left(\ln j\right)^2(1+o(1))}\right]\right)\right)}\\
&=e^{K-(1+o(1))\sum_{j=N+1}^{n}\frac{1-1/e}{2j}}=e^{-\ln n\frac{1-1/e+o(1)}{2}}=o(1).
\end{align*}
Part 2 follows.\\

Finally, let us prove that a.a.s. every vertex of $[N_0]$ has high degree. Let $j\in[N_0]$. For $n>N_0$, let $\xi_n$ be the Bernoulli random variable that equals 1 if and only if $n\sim j$ in $G_{n,m}$. Clearly, $\xi_{N_0+1},\xi_{N_0+2},\ldots$ are independent and $\Pr(\xi_n=1)=\frac{m}{n-1}$, $n>N_0$. Then, by Markov's inequality, for $n>N_0$,
\begin{multline*}
 \Pr(\xi_{N_0+1}+\ldots+\xi_n<K)=
 \Pr(e^{-(\xi_{N_0+1}+\ldots+\xi_n)}>e^{-K}) \\
\leq e^K\prod_{j=N_0+1}^n\left(1-(1-1/e)\frac{m}{j-1}\right)
=e^{-m(1-1/e+o(1))\ln n}=o(1).
\end{multline*}
Part 3 follows. 
\end{proof}

\section{Proof of Lemma~\ref{claim_m-game}}
\label{proof_claim_m-game}

The proof is based on the fact that Duplicator may play in a way such that, in the $r$-th round, for each $r\leq R$, the balls with radius $2^{R-r}$ and centers at pebbled vertices in one graph are similar (in some sense) to the respective balls in the other graph. This similarity for trees and unicyclic graphs can be easily defined by verifying the isomorphism between their spanning subgraphs obtained by some procedure defined in Section~\ref{trees_cycles}. %on isomorphism.
%In Section~\ref{trees_cycles}, we define this similarity.
The winning strategy of Duplicator is given in Section~\ref{win_strategy}. Although it is overloaded by technical details, the idea is quite simple. If a vertex pebbled by Spoiler in the $r$th round is far away from all the other pebbled vertices and from $[n_0]$, then the ball with its centre at this vertex contains at most one cycle. Duplicator pebbles a vertex that is also far enough from $[n_0]$ and all the remaining pebbled vertices, and with a similar $2^{R-r}$-neighborhood. If the vertex pebbled by Spoiler is far from all the other pebbled vertices but close to $[n_0]$, then $[N_0]$ divides its $2^{R-r}$-neighborhood into two parts. Duplicator chooses a vertex such that the intersection of its $2^{R-r}$-neighborhood with $[N_0]$ equals the intersection of the Spoiler's ball with $[N_0]$ and the remaining part (which is a forest) is similar to the rest of Spoiler's ball. Finally, if Spoiler chooses a vertex which is inside a $2^{R-r}$-neighborhood of a previously pebbled vertex, then the $2^{R-r}$-neighborhood of the respective pebbled vertex in the other graph is similar and, therefore, there is a suitable move for Duplicator inside this ball.

\subsection{Constructions}
\label{trees_cycles}

For a graph $G$ and its vertices $u,v$, we denote by $d_G(u,v)$ the {\it distance} between $u$ and $v$ (i.e., the length of a shortest path between $u$ and $v$ in $G$).\\

%In this section, we claim some properties of a local structure of the considered random graph models and prove the convergence laws via these properties. It turns out, that, after removing from the random graphs a small amount of initial vertices, we get a graph such that, for every $a\in\mathbb{N}$, with high probability an $a$-neighborhood of any vertex contains at most one cycle. Therefore, we are forced to study trees and unicyclic graphs that could be embedded in the random graphs as $a$-neighborhoods of certain vertices.

%Let $H\subset G$ be two graphs, and $H$ be a simple path. We call $H$ {\it $a$-separating} in $G$, if $H$ has at least $a$ vertices, and all its inner vertices have degree $2$ in $G$. 

%Let us call a connected graph {\it $a$-complex}, if it has at most $a$ vertices and at least $2$ simple cycles.

A perfect {\it $r$-ary tree} is a rooted tree where every non-leaf vertex has exactly $r$ children, and all leaf nodes are at the same distance from the root. {\it The depth} of a rooted tree is the longest distance between its root and a leaf.

Let us call an induced subgraph $H$ of $G$ {\it pendant}, if every vertex of $H$ having degree at least $2$ has no neighbors outside $H$.\\ % A rooted induced subtree $T$ of $G$ is {\it weakly pendant}, if the root of $T$ is adjacent to an only vertex $v$ outside $T$, and every non-root vertex of $T$ having a degree at least 2 has no neighbors outside $T$. The vertex $v$ is called {\it the parent of $T$}.\\

%A rooted tree is called {\it non-trivial} if it has more than one vertex, a non-rooted tree is called {\it non-trivial}, if it has more than 2 vertices, a unicyclic graph is called {\it non-trivial}, if it is not a simple cycle.\\

%We call a graph $H$ {\it super-admissible}, if, for $n$ large enough, with positive probability, $G_n$ has a pendant subgraph $\tilde H$ isomorphic to $H$. We call a rooted tree $T$ {\it super-admissible}, if, for $n$ large enough, with positive probability, $G_n$ has a rooted pendant subtree $\tilde T$, isomorphic to $T$.\\

%Clearly, every connected unicyclic graph is obtained by identifying every vertex of a simple cycle with a root of a rooted tree.

%Consider graphs $H\subset G$, where $H$ is a tree rooted in $R$. The subtree $H$ is called {\it pendant} if all its non-root vertices has no neighbors outside $H$, but the root has.

%Let $H$ be a unicyclic upper-admissible graph. Let $T_1,\ldots,T_a$ be all pendant subtrees of  $H$ (clearly, $a$ is the number of vertices in the only simple cycle in $H$).

% Let $\Sigma_{a,b}$ be a set of all (up to isomorphism) upper-admissible unicyclic graphs on $b$ vertices having $C_a$.

Fix $a\in\mathbb{N}$.
Let $T$ be a rooted tree of depth $d$. For $v\in V(T)$, let $T_v$ be a subtree rooted in a vertex $v$ of $T$ and induced on the set of all descendants of $v$ (children, children of their children, etc.)  and $v$ itself. A rooted tree $T^-_{a}$ is obtained from $T$ in the following $d$-step procedure. 

In step 1, consider vertices of $T$ at distance $d-1$ from the root. If such a vertex has more than $a$ children, remove all but $a$ of them. Denote the obtained graph by $T^1$.

Suppose $i\leq d-1$ many steps of the procedure have been completed. In step $i+1$, consider, one by one, every vertex of $T^i$ at distance $d-i-1$ from the root. For every such vertex $v$, consider the set $W_v$ of its children. Divide the set of trees $T^i_w$, $w\in W_v$, into isomorphism classes (of rooted trees). For every class, if its cardinality is greater than $a$, remove all but $a$ trees of this class from the tree. Denote the obtained graph by $T^{i+1}$.
Set $T^-_a=T^d$.

We say that two rooted trees $T_1,T_2$ are {\it $a$-isomorphic}, if $(T_1)^-_a\cong (T_2)^-_a$ (where by $\cong$ we denote the isomorphism of rooted trees). We say that $T$ is {\it $a$-trivial}, if $T$ is $a$-isomorphic to a perfect $a$-ary tree.\\

%A rooted tree $T$ is {\it $a$-super-admissible}, if there exists a super-admissible rooted tree $\tilde T$ which is $a$-isomorphic to $T$.\\

Let $C$ be a {\it rooted unicyclic graph of depth $d$} (it contains exactly one cycle and one vertex called the root, and the largest distance between the root and another vertex equals $d$) with root $R$ and %simple
cycle $C^*$. For $v\in V(C)$, let $T_v$ be a subtree of $C$ rooted in a vertex $v$ of $C$ and induced on the set of all {\it descendants} of $v$ ($u\neq v$ is a descendant of $v$ if any shortest path from $R$ to $u$ contains $v$, and any its vertex that follows after $v$ belongs to neither $C^*$ nor the shortest path between $R$ and $C^*$)  and $v$ itself. For every vertex $v$ either from $C^*$ or from the shortest path between $C^*$ and $R$, replace $T_v$ with $(T_v)^-_a$ (and preserve roots) and denote the obtained graph by $C^-_a$.

Let us call $C$ {\it perfect $a$-ary} if, for every vertex $v$ either from $C^*$ or from the shortest path between $C^*$ and $R$, $T_v$ is a perfect $a$-ary tree of depth $d-d_{C}(R,v)$.

We say that two rooted unicyclic graphs $C_1,C_2$ are {\it $a$-isomorphic}, if $(C_1)^-_a\cong (C_2)^-_a$ (the isomorphism preserves the root). We say that $C$ is {\it $a$-trivial}, if $C$ is $a$-isomorphic to a perfect $a$-ary unicyclic graph.\\

\subsection{The proof}
\label{win_strategy}

Consider graphs $H_1,H_2$ on vertex sets $[n_1]$ and $[n_2]$ respectively such that 
\begin{itemize}
\item their minimum degrees are at least $m$;
\item $H_1|_{[N_0]}= H_2|_{[N_0]}$;
\item $H_1$, $H_2$ have properties ${\sf Q1}$, ${\sf Q2}$, ${\sf Q3}$.
\end{itemize}

Without loss of generality, in the $(m-2)$-pebble game on $H_1$ and $H_2$, Spoiler chooses a vertex $x_1$ in $H_1$ in the first round. Duplicator responds with a vertex $y_1$ chosen by the following rules.

If $x_1\in[N_0]$ and $d_{H_1|_{[N_0]}}(x_1,[n_0])\leq 2^R$, then $y_1=x_1$.

If $d:=d_{H_1}(x_1,[n_0])\leq 2^R$ and either $x_1\notin[N_0]$ or $d_{H_1|_{[N_0]}}(x_1,[n_0])>2^R$, then find a shortest path $P$ between $x_1$ and a vertex from $[n_0]$. Let $v_0\ldots u$ be the longest subpath of $P$ that starts in $v_0\in[n_0]$ and never goes outside $[N_0]$. Let $\mathcal{B}_1$ be the ball of radius $2^R-d_{H_1}(u,x_1)$ in $H_1|_{[N_0]}$ with center in $u$. In $H_2$, there exists a vertex $y_1\notin\mathcal{B}_1$ such that the shortest path between $y_1$ and $u$ has length exactly $d_0=d_{H_1}(x_1,u)$ and all its inner vertices are outside $\mathcal{B}_1$. Indeed, by ${\sf Q3}$, $u$ has at least $m$ neighbors in $H_2$ outside $[N_0]$. Let $v$ be one of them. Let $v v_1\ldots v_{d_0-1}$ be an arbitrary simple path in $H_2$. If it meets $\mathcal{B}_1$, then either there is a cycle with length at most $2^R$ at distance at most $2^R-d_0\leq 2^R$ from $[n_0]$, or it meets $[n_0]$. In the latter case, we get a path joining two vertices of $[n_0]$ of length at most $2^m+d_0\leq 2^{m+1}$ and having the inner vertex $v$ outside $[N_0]$. This contradicts the property ${\sf Q1}$.

If $d>2^R$ and, for some $b$,  in $H_1$ there is a $b$-cycle $C_1$ inside the ball with radius $2^R$ and center in $x_1$, then, by ${\sf Q2}$, in $H_2$, there exists a $b$-cycle $C_2$ with all vertices outside $[N_0]$ and a vertex $y_1$ such that $d(y_1,C_2)=d(x_1,C_1)$. By the property ${\sf Q1}$, $d_{H_2}(y_1,[n_0])>2^R$.

Finally, if $d>2^R$ and, in $H_1$, there are no cycles inside the ball with radius $2^R$ and center in $x_1$, then let $v\notin[N_0]$ be a neighbor of a vertex $u$ from $[n_0]$. Find a path $uv\ldots y_1$ of length $2^R+1$. As above, by ${\sf Q1}$,  $d_{H_2}(y_1,[n_0])=2^R+1$ and, in $H_2$, there are no cycles inside the ball with radius $2^R$ and center in $y_1$.\\

Let us assume that $r\in[R-1]$ rounds are played. Without loss of generality, we may assume that all $m-2$ pairs of pebbles are placed on some vertices of $H_1$ and $H_2$. Let $x^r_i$ in $H_1$ and $y^r_i$ in $H_2$, $i\in[m-2]$, be the vertices occupied by the $i$-th pair of pebbles.
For $i\in[m-2]$, denote by $\mathrm{rd}_i^r$ the last round in which $x_i^r$ was pebbled. Moreover, assume that, for every $i\in[m-2]$, one of the following possibilities holds (it is clear that it holds for $r=1$): either $x^r_i$, $y^r_i$ are equal, belong to $[N_0]$ and are close to $[n_0]$ (this property is denoted by $1_i^r$ below), or both $x^r_i$ and $y^r_i$ are far from $[n_0]$ and their neighborhoods are trivial (the property $2_i^r$), or neighborhoods of $x^r_i$ and $y^r_i$ have equal intersections with $[N_0]$, and the deletion of these intersections transforms these neighborhoods into forests of trivial trees  (the property $3_i^r$). More formally,

\begin{itemize}
\item[$1_i^r$] $x_i^r=y_i^r\in[N_0]$, $d_{H_1|_{[N_0]}}(x_i^r,[n_0])\leq 2^{R+1-\mathrm{rd}_i^r}$;

\item[$2_i^r$] in $H_1$, there exists a pendant subgraph $\mathcal{B}^1_i$ which is either an $(m-1)$-trivial rooted tree of depth $2^{R+1-\mathrm{rd}_i^r}$ with root $x_i^r$ or an $(m-2)$-trivial unicyclic graph of depth $2^{R+1-\mathrm{rd}_i^r}$ with root $x_i^r$, such that $\mathcal{B}_i^1$ does not share vertices with $[n_0]$, and the same (existence of $\mathcal{B}^2_i$) applies for $y_i^r$ and $H_2$;

\item[$3_i^r$] there exists a subset $V_i\subset[N_0]$, a vertex $u_i^r\in V_i$, $(m-1)$-trivial rooted trees $T^1_i\subset H_1$, $T_i^2\subset H_2$ of depth $2^{R+1-\mathrm{rd}_i^r}$ and subgraphs $\mathcal{B}^1_i\subset H_1$, $\mathcal{B}^2_i\subset H_2$ such that

--- $x_i^r\neq u_i^r$, $y_i^r\neq u_i^r$ are the roots of $T^1_i$, $T^2_i$ respectively;

--- for $\lambda\in\{1,2\}$, $V(T_i^{\lambda})\cap V_i=\{u_i^r\}$;

--- $d_{T^1_i}(x_i^r,u_i^r)=d_{T^2_i}(y_i^r,u_i^r)$;

--- $v\in V_i$ if and only if $d_{H_1|_{[N_0]}}(v,u_i^r)\leq 2^{R+1-\mathrm{rd}_i^r}-d_{H_1}(x_i^r,u_i^r)$;

--- for $\lambda\in\{1,2\}$ and every non-leaf vertex $v\neq u_i^r$ of $T^{\lambda}_i$, $\mathrm{deg}_{T^{\lambda}_i}(v)=\mathrm{deg}_{H_{\lambda}}(v)$;

--- for $\lambda\in\{1,2\}$, $\mathrm{deg}_{H_{\lambda}}(u_i^r)=\mathrm{deg}_{T^{\lambda}_i}(u_i^r)+\mathrm{deg}_{H_1|_{V_i}}(u_i^r)$, i.e. $u_i^r$ does not lie on any edge other than edges from $T^{\lambda}_i$ and $H_1|_{V_i}$, and these two sets of edges are disjoint;

--- for $\lambda\in\{1,2\}$, $\mathcal{B}^{\lambda}_i=T^{\lambda}_i\cup H_{\lambda}|_{V_i}\cup\mathcal{F}^{\lambda}_i$ is the ball in $H_{\lambda}$ of radius $2^{R+1-\mathrm{rd}_i^r}$ and with center in $x_i^r$ or $y_i^r$ (for $\lambda=1$ or $\lambda=2$ resp.), $\mathcal{F}^{\lambda}_i$ is a forest of $(m-1)$-trivial rooted trees having roots in $V_i$ and sharing no other vertices with $V_i$ and $T^{\lambda}_i$.

\end{itemize}

In the case when either $2_i^r$ or $3_i^r$ holds, assume also that the following condition ${\sf IS}_i^r$ is satisfied (it roughly says that there exists an isomorphism between certain `representative' induced subgraphs of neighborhoods of $x_i^r$ and $y_i^r$ that preserves vertices at distance at most $2^{m+1-\mathrm{rd}_i^r}$ from $x_i^r$ and $y_i^r$ that were pebbled before the round $\mathrm{rd}_i^r$). 

\begin{itemize}

\item[${\sf IS}_i^r$] Let $\mathcal{Y}$ be the set of all $j\in[m-2]$ such that $\mathrm{rd}_j^r<\mathrm{rd}_i^r $ and $d(x_i^r,x_j^r)\leq 2^{m+1-\mathrm{rd}_i^r}$. Recall that $\mathcal{B}_i^1$ is the ball in $H_1$ with radius $2^{R+1-\mathrm{rd}_i^r}$ and center in $x_i^r$, and  $\mathcal{B}_i^2$ is the ball in $H_2$ with radius $2^{R+1-\mathrm{rd}_i^r}$ and center in $y_i^r$.

{\it If $2_i^r$ holds}, then there exist rooted graphs (either perfect $(m-1)$-ary trees or perfect $(m-2)$-ary unicyclic graphs) $\left(\mathcal{B}_i^1\right)^*\subset\mathcal{B}_i^1$, $\left(\mathcal{B}_i^2\right)^*\subset\mathcal{B}_i^2$ (here, the {\it induced} subgraph relations preserve roots) such that 

\begin{itemize}

\item $\mathcal{B}_i^1$ and $\left(\mathcal{B}_i^1\right)^*$, $\mathcal{B}_i^2$ and $\left(\mathcal{B}_i^2\right)^*$ are either $(m-1)$-isomorphic (in case of trees), or $(m-2)$-isomorphic (in case of unicyclic graphs), 

\item there exists an isomorphism $f:\left(\mathcal{B}_i^1\right)^*\to\left(\mathcal{B}_i^2\right)^*$ such that 

--- $f(x_j^r)=y_j^r$, $j\in\mathcal{Y}\cup\{i\}$,

--- for every $j\in\mathcal{Y}$ with $1_j^r$, either $x_i^r\in[N_0]$, $d_{H_1|_{[N_0]}}(x_j^r,x_i^r)\leq 2^{R-\mathrm{rd}_j^r}$ and $x_i^r=y_i^r$, or a shortest path from $x_j^r$ to $x_i^r$ has vertices outside $[N_0]$, leaves $[N_0]$ at the first time at vertex $u_j^r$, any shortest path from $y_j^r$ to $y_i^r$ leaves $[N_0]$ at the first time also at $u_j^r$ and $f(u_j^r)=u_j^r$.

\end{itemize}

{\it If $3_i^r$ holds}, then there exist perfect $(m-1)$-ary rooted trees $\left(R^{\lambda}_i\right)^*\subset R^{\lambda}_i$, $(F^{\lambda})^*\subset F^{\lambda}$ (the {\it induced} subgraph relations preserve roots) for all trees $F^{\lambda}$ from $\mathcal{F}^{\lambda}_i$ such that 

\begin{itemize}

\item $R_i^{\lambda}$ and $\left(R_i^{\lambda}\right)^*$, $F^{\lambda}$ and $(F^{\lambda})^*$ , $F^{\lambda}\in\mathcal{F}^{\lambda}_i$, are $(m-1)$-isomorphic, 

\item there exists an isomorphism  $f:\left(R^1_i\right)^*\cup\left(\mathcal{F}^1_i\right)^*\cup H_1|_{V_i}\to\left(R^2_i\right)^*\cup\left(\mathcal{F}^2_i\right)^*\cup H_2|_{V_i}$ (here, $\left(\mathcal{F}^{\lambda}_i\right)^*=\bigsqcup_{F^{\lambda}\in\mathcal{F}^{\lambda}_i}(F^{\lambda})^*$, where hereinafter $\sqcup$ denotes the disjoint union of sets) such that 

--- $f|_{\left(R^1_i\right)^*}:\left(R^1_i\right)^*\to \left(R^2_i\right)^*$ and $f|_{\left(\mathcal{F}^1_i\right)^*}:\left(\mathcal{F}^1_i\right)^*\to \left(\mathcal{F}^2_i\right)^*$ preserve roots,

--- $f(x_j^r)=y_j^r$, $j\in\mathcal{Y}\cup\{i\}$, 

--- $f(v)=v$ for $v\in V_i$.

\end{itemize}

Finally (in both cases: $2_i^r$ and $3_i^r$), let, for every $j\in[m-2]$ such that $\mathrm{rd}^r_j<\mathrm{rd}^r_i$, either $d(x^r_i,x^r_j)>2^{R+1-\mathrm{rd}_i^r}$ and $d(y^r_i,y^r_j)>2^{R+1-\mathrm{rd}_i^r}$, or $d(x_i^r,x_j^r)=d(y_i^r,y_j^r)\leq 2^{R+1-\mathrm{rd}_i^r}$.

\end{itemize}

% for every $1\leq i<j\leq r$, either $d(x_i,x_j)>2^{m+1-j}$ and $d(y_i,y_j)>2^{m+1-j}$, or $d(x_i,x_j)=d(y_i,y_j)\leq 2^{m+1-j}$. Assume also that the graphs induced on the sets of close chosen vertices and their heighborhoods are isomoorphic. In other words, 

%\begin{itemize}
%\item $1_i$, $1_j$. Then, $x_i=y_i$, $x_j=y_j$.

%\item $1_i$, $2_j$. Then, $x_i\in F_j^1$, $y_i\in F_j^2$, $\,F_j^1,F_j^2$ are trees.

%\item $1_i$, $3_j$. Then, $x_i=y_i\in V_j$.

%\item $2_i$, $2_j$. Then, $F_j^1\subset F_i^1$, $F_j^2\subset F_i^2$.

%\item $3_i$, $1_j$. Then, $x_j=y_j\in V_i$.

%\item $3_i$, $2_j$. Then, $F_j^1\subset G_i^1$, $F_j^2\subset G_i^2$, $\,F_j^1,F_j^2$ are trees.

%\item $3_i$, $3_j$. Then, $G_j^1\subset G_i^1$, $G_j^2\subset G_i^2$.
%\end{itemize}

%Finally, let the condition ${\sf IS}_r$ defined below is satisfied. Let $1\leq i_1<\ldots<i_{\ell}\leq r$ be such that, 

%--- for $1\leq \alpha<\beta\leq\ell$, $d(x_{i_{\alpha}},x_{i_{\beta}})=d(y_{i_{\alpha}},y_{i_{\beta}})\leq 2^{m+1-i_{\beta}}$, 

%--- $3_{i_1},\ldots, 3_{i_{\ell}}$ hold, 

%--- $u_{i_1}=\ldots=u_{i_{\ell}}$.

%Let, for $\gamma\in\{1,2\}$, $R^{\gamma}_*$ be obtained from $R^{\gamma}_{i_1}$ by moving the root from $x_{i_1}$ (or $y_{i_1}$) to $u_{i_1}$ and by removing all vertices $v$ such that $\left(R^{\gamma}_*\right)_v$ does not contain vertices $x_{i_1},\ldots,x_{i_{\ell}}$ (or $y_{i_1},\ldots,y_{i_{\ell}}$).

%Then, there exists an isomorphism of rooted trees $f:R^1_*\to R^2_*$ such that $f(x_{i_1})=y_{i_1},\ldots,f(x_{i_{\ell}})=y_{i_{\ell}}$.\\

Without loss of generality assume that, in round $r+1$, Spoiler moves the $(m-2)$th pebble from $x_{m-2}^r$ to $x_{m-2}^{r+1}$ in $H_1$. Set $x_i^r=x_i^{r+1}$, $y_i^r=y_i^{r+1}$ for all $i\in[m-3]$. Clearly, showing that there exists a vertex $y_{m-2}^{r+1}$ in $H_2$ such the above (for every $i\in[m-2]$, one of the following three possibilities: either $1_i^{r+1}$, or $2_i^{r+1}$ and $\mathrm{IS}_i^{r+1}$, or $3_i^{r+1}$ and $\mathrm{IS}_i^{r+1}$) also holds for the round $r+1$, finishes the proof of Lemma~\ref{claim_m-game}. Indeed, if this is true, then by induction, we get that, in the last round $R$, for every $i\in[m-2]$, one of the mentioned three possibilities holds. Then, fix distinct $i,j\in[m-2]$. If $x_i^R$ and $x_j^R$ both have the property $1_i^R$, then $(x_i^R\sim x_j^R)\Leftrightarrow(y_i^R\sim y_j^R)$ since $x_i^R=y_i^R$, $x_j^R=y_j^R$ are inside the same induced subgraph $H_1|_{[N_0]}=H_2|_{[N_0]}$ of both $H_1$ and $H_2$. Assume without loss of generality that $\mathrm{rd}_i^R>\mathrm{rd}_j^R$. If $1_i^R$ holds but $1_j^R$ does not hold, then either $d(x_j^R,[n_0])>2^{R+1-\mathrm{rd}_j^R}$ while $d(x_i^R,[n_0])\leq 2^{R+1-\mathrm{rd}_i^R}$  (and then $x_i^R\nsim x_j^R$, $y_i^R\nsim y_j^R$), or the property $3_j^R$ holds. In the latter case, either $x_i^R=y_i^R$ does not belong to $V_j$, and then $x_i^R\nsim x_j^R$, $y_i^R\nsim y_j^R$, or $x_i^R=y_i^R\in V_i$, and then $(x_i^R\sim x_j^R)\Leftrightarrow(y_i^R\sim y_j^R)$ since, by the definition of $3_j^R$, $d_{T_j^1}(x_j^R,u_j^R)=d_{T_j^2}(y_j^R,u_j^R)$. If either $2_i^R$ or $3_i^R$ holds, then $(x_i^R\sim x_j^R)\Leftrightarrow(y_i^R\sim y_j^R)$ due to the last condition in the definition of the property  $\mathrm{IS}_i^{R}$: either $d(x^R_i,x^rRj)>2^{R+1-\mathrm{rd}_i^r}$ and $d(y^R_i,y^R_j)>2^{R+1-\mathrm{rd}_i^r}$, or $d(x_i^R,x_j^R)=d(y_i^R,y_j^R)\leq 2^{R+1-\mathrm{rd}_i^r}$.\\

 Let us now prove the step of induction.

\begin{enumerate}

\item If $x^{r+1}_{m-2}\in[N_0]$ and $d_{H_1|_{[N_0]}}(x^{r+1}_{m-2},[n_0])\leq 2^{R-r}$, then $y_{m-2}^{r+1}=x_{m-2}^{r+1}$. So, $1_{m-2}^{r+1}$ holds. 

It remains to prove that, for $j\in[m-3]$, either $d(x_j^{r+1},x_{m-2}^{r+1})>2^{R-r}$ and $d(y_j^{r+1},y_{m-2}^{r+1})>2^{R-r}$, or $d(x_j^{r+1},x_{m-2}^{r+1})=d(y_j^{r+1},y_{m-2}^{r+1})\leq 2^{R-r}$.\\

$\bullet\quad$ Assume that $d(x_j^{r+1},x_{m-2}^{r+1})>2^{R-r}$. 

If $1_j^{r+1}$ holds, then $d_{H_2|_{[N_0]}}(y_j^{r+1},y_{m-2}^{r+1})=d_{H_1|_{[N_0]}}(x_j^{r+1},x_{m-2}^{r+1})>2^{R-r}$. %Notice that $d_{H_2|_{[N_0]}}(y_j,y_{r+1})\leq 2^{m+1-i}+2^{m-r}$. 
If, in $H_2$, there exists a path between $y_j^{r+1}$ and $y_{m-2}^{r+1}$ having length at most $2^{R-r}$, then it has a vertex outside $[N_0]$, that contradicts ${\sf Q1}$ (indeed, $d_{H_1|_{[N_0]}}(y_j^{r+1},y_{m-2}^{r+1})\leq d_{H_1|_{[N_0]}}(y_j^{r+1},[n_0])+d_{H_1|_{[N_0]}}(y_{m-2}^{r+1},[n_0])\leq 2^{R-r}+2^{R+1-\mathrm{rd}_j^{r+1}}$). Therefore, $d_{H_2}(y_j^{r+1},y_{m-2}^{r+1})>2^{R-r}$.

If $2_j^{r+1}$ holds, then $\mathcal{B}^2_j$ does not contain any vertex of $[n_0]$. Therefore, $d(y_j^{r+1},y_{m-2}^{r+1})\geq d(y_j^{r+1},[n_0])-d(y_{m-2}^{r+1},[n_0])>2^{R-r}$. 

If $3_j^{r+1}$ holds and $x_{m-2}^{r+1}\in V_j$, then 
\begin{equation}
\begin{split}
d(y_j^{r+1},y_{m-2}^{r+1})&=d(y_j^{r+1},u_j^{r+1})+d(u_j^{r+1},y_{m-2}^{r+1})\\
&=d(x_j^{r+1},u_j^{r+1})+d(u_j^{r+1},x_{m-2}^{r+1})\\
&=d(x_j^{r+1},x_{m-2}^{r+1})>2^{R-r}.
\end{split}
\label{from_x_to_y}
\end{equation}
Finally, if $x_{m-2}^{r+1}\notin V_j$, then either $y_{m-2}^{r+1}$ does not belong to $\mathcal{B}^2_j$ and, therefore, $d(y_j^{r+1},y_{m-2}^{r+1})\geq 2^{R+1-\mathrm{rd}_j^{r+1}}>2^{R-r}$, or $y_{m-2}^{r+1}\in V(\mathcal{B}^2_j)$. Let $d(y_j^{r+1},y_{m-2}^{r+1})\leq 2^{R-r}$. If $y_{m-2}^{r+1}$ belongs to the `tree part' of $\mathcal{B}^2_j$, then there exists a path $y_{m-2}^{r+1}\ldots y_j^{r+1}\ldots u_{j}^{r+1}\ldots[n_0]$ of length at most $2^{R-r}+3\cdot 2^{R+1-\mathrm{rd}_j^{r+1}}$ with at least one vertex outside $[N_0]$. Together with the condition $d_{H_2|_[N_0]}(y_{m-2}^{r+1},[n_0])$, it contradicts the property ${\sf Q1}$. If $y_{m-2}^{r+1}$ belongs to the `forest part' of $\mathcal{B}^2_j$, then denoting by $f$ the root of the tree that $y_{m-2}^{r+1}$ belongs to, we get $d(y_j^{r+1},y_{m-2}^{r+1})=d(y_{m-2}^{r+1},f)+d(f,u_j^{r+1})+d(u_j^{r+1},y_j^{r+1})=d(x_j^{r+1},x_{m-2}^{r+1})$ --- a contradiction.\\

% and, therefore, the shortest path between $y_i$ and $y_{r+1}$ does not meet $u_i$. In the latter case, if $d(y_i,y_{r+1})\leq 2^{m-r}$, then, since $d(y_{r+1},[n_0])\leq 2^{m-r}$, there is either a simple cycle of length at most $2^{m+1}$ at distance at most $2^m$ from $[n_0]$, or a path joining two vertices of $[n_0]$ of length at most $2^{m+1}$ and having a vertex outside $[N_0]$. Thus, $d(y_i,y_{r+1})>2^{m-r}$.\\

$\bullet\quad$ Assume that $d(x_j^{r+1},x_{m-2}^{r+1})\leq 2^{R-r}$. 

Then, $d(x_j^{r+1},[n_0])\leq 2^{R-r+1}\leq 2^{R-\mathrm{rd}_j^{r+1}+1}$. Also, if $x_j^{r+1}\in[N_0]$ and $d_{H_1|_{[N_0]}}(x_j^{r+1},[n_0])\leq 2^{R-\mathrm{rd}_j^{r+1}+1}$, then $x_j^{r+1}=y_j^{r+1}$ and, therefore, we get $d(y_j^{r+1},y_{m-2}^{r+1})=d(x_j^{r+1},x_{m-2}^{r+1})$ due to the property ${\sf Q1}$. %Indeed, if, in $H_1$, there exists a path between $x_i$ and $x_{r+1}$ with length at most $2^{m-r}$ and a vertex outside $[N_0]$, then there is either a simple cycle with length at most $2^{m+1}$ at distance at most $2^m$ from $[n_0]$, or a path joining two vertices of $[n_0]$ of length at most $2^{m+1}$ and having a vertex outside $[N_0]$, and the same applies to $H_2$. A contradiction. In this case, $1_i$ and $1_{r+1}$ hold.
If $x_j^{r+1}\notin[N_0]$ or $d_{H_1|_{[N_0]}}(x_j^{r+1},[n_0])>2^{R-\mathrm{rd}_j^{r+1}+1}$, then $3_j^{r+1}$ holds. If $x_{m-2}^{r+1}\in V_j$, then all the equalities from (\ref{from_x_to_y}) hold as well. 
%$$
%d(y_j^{r+1},y_{m-2}^{r+1})=d(y_j^{r+1},u_j^{r+1})+d(u_j^{r+1},y_{m-2}^{r+1})=d(x_j^{r+1},u_j^{r+1})+d(u_j^{r+1},x_{m-2}^{r+1})=d(x_j^{r+1},x_{m-2}^{r+1}).
%$$
If $x_{m-2}^{r+1}\notin V_j$, then $x_{m-2}^{r+1}$ belongs either to `the tree' or to `the forest part' of $\mathcal{B}_j^1$. But this contradicts ${\sf Q1}$, since we get two simple paths connecting $x_{m-2}^{r+1}$ with $[n_0]$ of lengths at most $2^{R-\mathrm{rd}_j^{r+1}+1}+2^{R-r}$ such that exactly one of them has vertices outside $[N_0]$.\\

% does not belong to $\mathcal{B}^1_j$ due to the property ${\sf Q1}$. % since, otherwise, there is either a simple cycle with length at most $2^{m+1}$ at distance at most $2^m$ from $[n_0]$, or a path joining two vertices of $[n_0]$ of length at most $2^{m+1}$ and having a vertex outside $[N_0]$. 
%But this is impossible, since $\mathcal{B}^1_j$ is the ball in $H_1$ with center in $x_j^{r+1}$ and radius $2^{m+1-j}>2^{R-r}$ and $d(x_j^{r+1},x_{m-2}^{r+1})\leq 2^{R-r}$.

\item Let $d(x_{m-2}^{r+1},[n_0])\leq 2^{R-r}$ and either $x_{m-2}^{r+1}\notin[N_0]$ or $d_{H_1|_{[N_0]}}(x_{m-2}^{r+1},[n_0])>2^{R-r}$.\\

Let $\mathcal{J}$ be the set of all $j\in[m-3]$ such that $d(x_j^{r+1},x_{m-2}^{r+1})\leq 2^{R-r}$. Divide the set $\mathcal{J}$ in the following way $\mathcal{J}=\mathcal{J}_1\sqcup\mathcal{J}_3$: $j\in\mathcal{J}_1$ if and only of $1_j^{r+1}$ holds and $j\in\mathcal{J}_3$ if and only if $3_j^{r+1}$ holds.\\

$\bullet\quad$ Assume first that $\mathcal{J}_3$ is empty.

%Assume that there is no $i\in[r]$ such that $d(x_i,x_{r+1})\leq 2^{m-r}$. 

The way how $y_{m-2}^{r+1}$ is chosen is similar to the way how $y_1$ is chosen in the first round. The only difference is that we should find a vertex which is far enough from all the chosen vertices $y_j^{r+1}$ with $3_j^{r+1}$. 

Recall that $u_{m-2}^{r+1}$ is the vertex of $[N_0]$ after which a shortest path joining $[n_0]$ with $x_{m-2}^{r+1}$ leaves the set $[N_0]$ at the first time. The set $V_{m-2}$ induces the ball of radius $2^{R-r}-d_{H_1}(u_{m-2}^{r+1},x_{m-2}^{r+1})$ in $H_1|_{[N_0]}$ with center in $u_{m-2}^{r+1}$. 

Find the set $\mathcal{I}$ of all $j\in[m-3]$ such that $3_j^{r+1}$ holds and $u_j^{r+1}=u_{m-2}^{r+1}$. For $j\in\mathcal{I}$, let $v_j$ be the neighbor of $u_j^{r+1}$ on the path between $u_j^{r+1}$ and $y_j^{r+1}$ in $R^2_j$. Due to the properties ${\sf Q1}$ and ${\sf Q3}$, in $H_2$, there exists a vertex $y_{m-2}^{r+1}\notin V_{m-2}$ such that the shortest path between $y_{m-2}^{r+1}$ and $u_{m-2}^{r+1}$ has length exactly $d_0=d_{H_1}(x_{m-2}^{r+1},u_{m-2}^{r+1})$, all its inner vertices are outside $V_{m-2}$ and the neighbor of $u_{m-2}^{r+1}$ in this path does not belong to $\{v_j,j\in\mathcal{I}\}$. %Let $v v_1\ldots v_{d_0-1}$ be an arbitrary simple path in $H_2$. If it meets $\mathcal{B}_1$, then we get a contradiction with the property ${\sf Q1}$ in the same way as in the first round.

The properties ${\sf Q1}$ and ${\sf Q3}$ imply $3_{m-2}^{r+1}$. For $j\in\mathcal{J}_1$, $x_j^{r+1}=y_j^{r+1}\in V_{m-2}$ since, otherwise, we get a contradiction with the property ${\sf Q1}$ in the usual way. Therefore, $d(y_j^{r+1},y_{m-2}^{r+1})=d(x_j^{r+1},x_{m-2}^{r+1})$.

It remains to prove that there is no $j\in[m-3]\setminus\mathcal{J}_1$ such that $d(y_j^{r+1},y_{m-2}^{r+1})\leq 2^{R-r}$. Assume that such a $y_j^{r+1}$, $j\in[m-3]\setminus\mathcal{J}_1$, exists. Clearly, $y_j^{r+1}\notin V_{m-2}$ since, otherwise, $d_{H_2|_{[N_0]}}(y_j^{r+1},[n_0])< 2^{R-r}<2^{R+1-\mathrm{rd}_j^{r+1}}$ and, therefore, $x_j^{r+1}=y_j^{r+1}$. This contradicts the assumption that $d(x_j^{r+1},x_{m-2}^{r+1})>2^{R-r}$ since
\begin{equation}
\begin{split}
d(x_j^{r+1},x_{m-2}^{r+1})&=d(x_j^{r+1},u_{m-2}^{r+1})+d(u_{m-2}^{r+1},x_{m-2}^{r+1})\\
&=d(y_j^{r+1},u_{m-2}^{r+1})+d(u_{m-2}^{r+1},y_{m-2}^{r+1})\\
&=d(y_j^{r+1},y_{m-2}^{r+1})\leq 2^{R-r}.
\end{split}
\label{from_y_to_x}
\end{equation}
If $y_j^{r+1}\in[N_0]\setminus V_{m-2}$ and $d_{H_2|_{[N_0]}}(y_j^{r+1},[n_0])\leq 2^{R+1-\mathrm{rd}_j^{r+1}}$, then $y_j^{r+1}=x_j^{r+1}$. By ${\sf Q1}$, all the equalities from (\ref{from_y_to_x}) hold as well. Since $j\notin\mathcal{J}_1$, we get $d(y_j^{r+1},y_{m-2}^{r+1})>2^{R-r}$ --- a contradiction. Thus, $3_j^{r+1}$ holds. Let $u_j^{r+1}\neq u_{m-2}^{r+1}$. Since $d(y_j^{r+1},y_{m-2}^{r+1})\leq 2^{R-r}$, the shortest path between $y_{m-2}^{r+1}$ and $y_j^{r+1}$ goes through $u_j^{r+1}$ and $u_{m-2}^{r+1}$ (otherwise, we get a contradiction with ${\sf Q1}$). Therefore, $u_{m-2}^{r+1}\in V_j$ and

\begin{align*}
d(x_j^{r+1},x_{m-2}^{r+1})&\leq d(x_j^{r+1},u_j^{r+1})+d(u_j^{r+1},u_{m-2}^{r+1})+d(u_{m-2}^{r+1},x_{m-2}^{r+1})\\
&=d(y_j^{r+1},u_j^{r+1})+d(u_j^{r+1},u_{m-2}^{r+1})+d(u_{m-2}^{r+1},y_{m-2}^{r+1})\\
&=d(y_j^{r+1},y_{m-2}^{r+1})\leq 2^{R-r}
\end{align*}

--- a contradiction. Finally, let $u_j^{r+1}=u_{m-2}^{r+1}$. Since, by the construction, the neighbors of $u_{m-2}^{r+1}$ in the paths between $u_{m-2}^{r+1}$ and $y_j^{r+1}$, $y_{m-2}^{r+1}$ are distinct, then, due to the property ${\sf Q1}$, the shortest path between $y_j^{r+1}$ and $y_{m-2}^{r+1}$ is the union of the paths between $u_j^{r+1},y_j^{r+1}$ and $u_j^{r+1},y_{m-2}^{r+1}$. Then,
\begin{align*}
d(x_j^{r+1},x_{m-2}^{r+1})&\leq d(x_j^{r+1},u_j^{r+1})+d(u_{m-2}^{r+1},x_{m-2}^{r+1})\\
&=d(y_j^{r+1},u_j^{r+1})+d(u_{m-2}^{r+1},y_{m-2}^{r+1})\\
&=d(y_j^{r+1},y_{m-2}^{r+1})\leq 2^{R-r}
\end{align*}
--- a contradiction.\\

$\bullet\quad$ Now, let $\mathcal{J}_3\neq\varnothing$. Let $j\in\mathcal{J}_3$ be such that $\mathrm{rd}_j^{r+1}$ is the maximum in $\{\mathrm{rd}_i^{r+1},\,i\in\mathcal{J}_3\}$. %Notice that, for every other $i\in\mathcal{J}_3$,
%$$
% d(x_j,x_i)\leq d(x_j,x_{r+1})+d(x_i,x_{r+1})\leq 2^{m+1-r}\leq 2^{m+1-j}.
%$$
%$G_j^1\subset G_i^1$, $G_j^2\subset G_i^2$.

%Let us find $y_{r+1}$ such that $3_{r+1}$ holds and $G_{r+1}^2\subset G_j^2$.

Due to the property ${\sf Q1}$, either $x_{m-2}^{r+1}\in V(R^1_j)$, or $x_{m-2}^{r+1}\in V(\mathcal{F}^1_j)$. In the latter case, $u_{m-2}^{r+1}$ is the root of a tree from $\mathcal{F}^1_j$ that contains $x_{m-2}^{r+1}$. However, we will assume that $x_{m-2}^{r+1}\in V(R^1_j)$ to avoid new notations (this does not change the below arguments anyhow). Find all $i\in\mathcal{J}_3$ such that $u_i^{r+1}=u_{m-2}^{r+1}$. Let $\mathcal{J}_3^0$ be the set of all such $i$. Let $R^1_*$ be the tree obtained from $R^1_{m-2}$ by moving the root to $u_{m-2}^{r+1}$, and removing all vertices $v\notin\{u_{m-2}^{r+1},\,x_i^{r+1},i\in\mathcal{J}_3^0\}$ such that, in the obtained tree, there is no descendant of $v$ equal to any of $x_i^{r+1}$, $i\in\mathcal{J}_3^0$. By the condition ${\sf IS}_j^{r+1}$, $R^2_j$ contains $R^2_*\cong R^1_*$ rooted in $u_{m-2}^{r+1}$ and there exists an isomorphism of rooted trees $f:R^1_*\to R^2_*$ such that $f(x_i^{r+1})=y_i^{r+1}$, $i\in\mathcal{J}_3^0$. 

If $x_{m-2}^{r+1}\in V(R^1_*)$, then set $f(x_{m-2}^{r+1})=y_{m-2}^{r+1}$. If $x_{m-2}^{r+1}\notin V(R^1_*)$, then find the closest vertex $v_1$ of $R^1_*$ to $x_{m-2}^{r+1}$ in $R^1_j$. Let $v_2=f(v_1)$. Find a neighbor $v\notin V(R^2_*)\cup\{y_i^{r+1},\,i\in[m-3]\}$ of $v_2$ in $R^2_j$ such that, in the tree obtained from $R^2_j$ by moving the root to $u_{m-2}^{r+1}$, there is no descendant of $v$ among $y_i^{r+1}$, $i\in[m-3]$ (the existence of such a neighbor follows from the fact that the minimum degree of $H_2$ is at least $m$). Let $v_2 v\ldots y_{m-2}^{r+1}$ be a path in $R^2_j$ with length $d_{R^1_j}(v_1,x_{m-2}^{r+1})$.% and having the only common vertex $v_2$ with $R^2_*$.

%[Fix this part - we should avoid some vertices. Proof that this vertex is a desired one]

% a neighbor $v$ of $y_i$ in $R^2_j$ such that $(R^2_j)_v$ does not contain any vertex $y_i$, $i\in\mathcal{J}_1$. Let $y_{r+1}$ be a vertex of $(R^2_j)_v$ at distance $d_{R^1_j}(x_i,x_{r+1})$ from $y_i$. In the latter case, let $f$ be the root of the tree from $\mathcal{F}^1_j$ that contains $x_{r+1}$. Let $T$ be the tree from $\mathcal{F}^2_j$ rooted in $f$. Find a neighbor $v$ of $T$ such that $T_v$ does not contain any vertex $y_i$, $i\in\mathcal{J}_1$. Such a vertex exists since $f\in[N_0]$ and, therefore, it has at least $m$ neighbours outside $[N_0]$. Let $y_{r+1}$ be a vertex of $T_v$ such that $d_{\mathcal{F}^2_j}(y_{r+1},f)=d_{\mathcal{F}^1_j}(x_{r+1},f)$.

It is clear that $3_{m-2}^{r+1}$ holds. To prove ${\sf IS}_{m-2}^{r+1}$, it remains to show that, for all $i\notin\mathcal{J}$, $d(y_i^{r+1},y_{m-2}^{r+1})>2^{R-r}$ and, for all $i\in\mathcal{J}_1$, $d(y_i^{r+1},y_{m-2}^{r+1})=d(x_i^{r+1},x_{m-2}^{r+1})$.

First, let $i\notin\mathcal{J}$. If $y_i^{r+1}\notin V(\mathcal{B}_{m-2}^2)$, then $d(y_{m-2}^{r+1},y_i^{r+1})>2^{R-r}$. Let $y_i\in V(\mathcal{B}_{m-2}^2)$, then, by the construction, $x_i^{r+1}\in V(\mathcal{B}_{m-2}^1)$ and $i\in\mathcal{J}$ --- a contradiction.

Finally, let $i\in\mathcal{J}_1$. Then, $y_i^{r+1}=x_i^{r+1}\in V_{m-2}$, since, otherwise, ${\sf Q1}$ does not hold. Therefore, 
\begin{equation*}
\begin{split}
d(y_{m-2}^{r+1},y_i^{r+1})&=
d(y_{m-2}^{r+1},u_{m-2}^{r+1})+d(u_{m-2}^{r+1},y_i^{r+1})\\
&=d(x_{m-2}^{r+1},u_{m-2}^{r+1})+d(u_{m-2}^{r+1},x_i^{r+1})\\
&=d(x_{m-2}^{r+1},x_i^{r+1}).
\end{split}
\end{equation*}

%Finally, let $i\in\mathcal{J}_3$. If $d(x_i,x_{r+1})>2^{m-r}$, then $x_i\notin V(G_{r+1}^1)$. Assume that $d(y_i,y_{r+1})\leq 2^{m-r}$. By the construction, $d(y_i,y_{r+1})=d(y_i,v_2)+d(v_2,y_{r+1})$ (here, $v_2=y_{r+1}$ if $y_{r+1}\in R^2_*$). Then $d(v_2,y_i)\leq 2^{m-r}$. Since $v_2\in R^2_*$, we get that $u_{r+1}=u_i$. By the condition ${\sf IS}_r$, $d(v_2,y_i)=d(v_1,x_i)$. Therefore, 
%$$
%d(x_i,x_{r+1})\leq d(x_i,v_1)+d(v_1,x_{r+1)}=d(y_i,v_2)+d(v_2,y_{r+1})=d(y_i,y_{r+1})\leq 2^{m-r}
%$$
%--- a contradiction. If $d(x_i,x_{r+1})\leq 2^{m-r}$, then $y_{r+1}=f(x_{r+1})$ belongs to $R^2_*$ isomorphic to $R^1_*$ and, therefore, $d(y_i,y_{r+q})=d(x_i,x_{r+1})$.

\item Let $d(x_{m-2}^{r+1},[n_0])> 2^{R-r}$. Let $\mathcal{J}$ be the set of all $j\in[m-3]$ such that $d(x_j^{r+1},x_{m-2}^{r+1})\leq 2^{R-r}$. Divide the set $\mathcal{J}$ in the following way $\mathcal{J}=\mathcal{J}_1\sqcup\mathcal{J}_2\sqcup\mathcal{J}_3$: $j\in\mathcal{J}_{\chi}$ if and only of $\chi_j^{r+1}$ holds, $\chi\in\{1,2,3\}$. \\

$\bullet\quad$ Assume first that $\mathcal{J}$ is empty. \\

If, in $H_1$, for some $b$, there is a $b$-cycle $C^1$ inside the ball with radius $2^{R-r}$ and center in $x_{m-2}^{r+1}$, then, by ${\sf Q1}$ and ${\sf Q2}$, in $H_2$, there exists a vertex $y_{m-2}^{r+1}$ and a $b$-cycle $C^2$ such that $d(y_{m-2}^{r+1},C^2)=d(x_{m-2}^{r+1},C^1)$ and $d(y_{m-2}^{r+1},y_j^{r+1})>2^{R-r}$ for all $j\in[m-3]$.\\

If, in $H_1$, there are no cycles inside the ball with radius $2^{R-r}$ and center in $x_{m-2}^{r+1}$, then fix a vertex $u\in[n_0]$ and find all its neighbors in $H_2$ outside $[N_0]$. By the property ${\sf Q3}$, there are at least $m$ such neighbors. Then, by the property ${\sf Q1}$, at least one of them ($v$) is such that any simple path of length $2^{R-r+1}$ that starts on $v$ and does not meet $u$ does not contain any of $y_j^{r+1}$, $j\in[m-3]$. Let $y_{m-2}^{r+1}$ be a vertex in the middle of one of such paths. Clearly, for every $j\in[m-3]$, $d(y_{m-2}^{r+1},y_j^{r+1})>2^{R-r}$ and, by the property ${\sf Q1}$, there are no cycles inside the ball with radius $2^{R-r}$ and center in $y_{m-2}^{r+1}$.\\

$\bullet\quad$ Let $\mathcal{J}_1\neq\varnothing$ and $\mathcal{J}_2\sqcup\mathcal{J}_3=\varnothing$.\\ %Clearly, $\mathcal{B}^1_{r+1}$ is a tree. 

Assume first that there is no $j\in\mathcal{J}_1$ such that a shortest path between $x_j^{r+1}$ and $x_{m-2}^{r+1}$ lies inside $[N_0]$. Let $j\in\mathcal{J}_1$. Let $P$ be a shortest path from $x_j^{r+1}$ to $x_{m-2}^{r+1}$ and let $u$ be the vertex from $[N_0]$ where the path leaves $[N_0]$ at the first time. Let $V_{m-2}$ be the set of all vertices of $[N_0]$ such that shortest paths between them and $u$ lie in $[N_0]$ and are not longer than $2^{R-r}-d(x_{m-2}^{r+1},u)$. Find all neighbors of $u$ in $H_2$ outside $[N_0]$. By the property ${\sf Q3}$, there are at least $m$ such neighbors. Then, by the property ${\sf Q1}$, at least one of them ($v$) is such that any simple path of length $2^{R-r+1}$ that starts on $v$ and does not meet $u$ does not contain any of $y_i^{r+1}$, $i\in[m-3]$. Let $y_{m-2}^{r+1}$ be a vertex of one of such paths such that $d(y_{m-2}^{r+1},u)=d(x_{m-2}^{r+1},u)$. Clearly $d(y_{m-2}^{r+1},y_j^{r+1})=d(x_{m-2}^{r+1},x_j^{r+1})$. Moreover, by the property ${\sf Q1}$, the ball $\mathcal{B}^2_{m-2}$ with radius $2^{R-r}$ and center in $y_{m-2}^{r+1}$ contains $V_{m-2}$ and $d(y_{m-2}^{r+1},V_{m-2})=d(x_{m-2}^{r+1},V_{m-2})$. By the property ${\sf Q1}$, any $x_i^{r+1}$, $i\in\mathcal{J}_1$, should be inside $V_{m-2}$, and $d(x_i^{r+1},x_{m-2}^{r+1})=d(x_{m-2}^{r+1},u)+d(u,x_i^{r+1})$. Therefore, for all such $i$, $d(y_i^{r+1},y_{m-2}^{r+1})=d(x_i^{r+1},x_{m-2}^{r+1})$. It remains to prove that, for all $i\in[m-3]\setminus\mathcal{J}_1$, $d(y_i^{r+1},y_{m-2}^{r+1})>2^{R-r}$. By the construction, any $y_i$ such that $d(y_i^{r+1},y_{m-2}^{r+1})\leq 2^{R-r}$ should be either inside $V_{m-2}$, or connected by a simple path $\tilde P$ with a vertex $\tilde u\in V_{m-2}$ such that the neighbor of $\tilde u$ in this path does not belong to $[N_0]$. In the first case, $1_i^{r+1}$ holds. In the latter case, $3_i^{r+1}$ holds. But this is impossible since $i\notin\mathcal{J}$.\\

%. But there are no such $y_i^{r+1}$, $i\in[r]\setminus\mathcal{J}_1$.

Now, let there exist $j\in\mathcal{J}_1$ such that a shortest path between $x_j^{r+1}$ and $x_{m-2}^{r+1}$ lies inside $[N_0]$. Set $y_{m-2}^{r+1}=x_{m-2}^{r+1}$. Clearly, we only need to prove that there are no $i\in [m-3]\setminus\mathcal{J}_1$ such that $d(y_i^{r+1},y_{m-2}^{r+1})\leq 2^{R-r}$. Assume, that for some $i\in [m-3]\setminus\mathcal{J}_1$, $d(y_i^{r+1},y_{m-2}^{r+1})\leq 2^{R-r}$. If $1_i$ holds, then a shortest path between $x_i^{r+1}=y_i^{r+1}$ and $x_{m-2}^{r+1}=y_{m-2}^{r+1}$ has at least one vertex outside $[N_0]$ (otherwise, $d(x_i^{r+1},x_{m-2}^{r+1})=d(y_i^{r+1},y_{m-2}^{r+1})\leq 2^{R-r}$). But then, since $d(x_i^{r+1},[n_0])\leq 2^{R+1-\mathrm{rd}_i^{r+1}}$ and $d(x_j^{r+1},[n_0])\leq 2^{R+1-\mathrm{rd}_j^{r+1}}$, we get a contradiction with ${\sf Q1}$. If $2_i^{r+1}$ or $3_i^{r+1}$ holds, then consider a shortest path $P$ in $H_1$ from $[n_0]$ to $x_i^{r+1}$. Let $u_i^{r+1}$ be the first vertex of $[N_0]$ after which $P$ leaves $[N_0]$. Since $d(y_j^{r+1},y_i^{r+1})\leq d(y_j^{r+1},y_{m-2}^{r+1})+d(y_i^{r+1},y_{m-2}^{r+1})\leq 2^{R-r+1}$, we get that $d(x_j^{r+1},x_i^{r+1})\leq 2^{R-r+1}$ as well by the induction hypothesis. Then, if $2_i^{r+1}$ holds, a shortest path in $H_2$ between $[n_0]$ and $y_i^{r+1}$ leaves $[N_0]$ at the same vertex due to the property ${\sf IS}_i^{r+1}$. And the same applies if $3_i^{r+1}$ holds by the definition of this property. Moreover, a shortest path from $y_{m-2}^{r+1}$ to $y_i^{r+1}$ also leaves $[N_0]$ at $u_i^{r+1}$ due to the property ${\sf Q1}$. Since $x_{m-2}^{r+1}=y_{m-2}^{r+1}$ and due to ${\sf Q1}$, any shortest path between $x_{m-2}^{r+1}$ and $x_i^{r+1}$ meets $u_i^{r+1}$. Therefore, we get $d(y_{m-2}^{r+1},y_i^{r+1})=d(y_{m-2}^{r+1},u_i^{r+1})+d(u_i^{r+1},y_i^{r+1})=d(x_{m-2}^{r+1},u_i^{r+1})+d(u_i^{r+1},x_i^{r+1})=d(x_{m-2}^{r+1},x_i^{r+1})$. But then $i\in\mathcal{J}$ --- a contradiction.\\

$\bullet\quad$ Let $\mathcal{J}_2\sqcup\mathcal{J}_3\neq\varnothing$. Let $j\in\mathcal{J}_2\sqcup\mathcal{J}_3$ be such that $\mathrm{rd}_j^{r+1}$ is maximum in $\{\mathrm{rd}_i^{r+1},i\in\mathcal{J}_2\sqcup\mathcal{J}_3\}$.\\

Let $\mathcal{B}^1_{m-2}$ contain a cycle $C$.%, then $\mathcal{J}_1=\mathcal{J}_3=\varnothing$. 

Let $\mathcal{B}^1_*$ be obtained from $\mathcal{B}^1_{m-2}$ by making a vertex of $C$ the root and removing all vertices $v\notin V(C)\cup\{x_i^{r+1},i\in\mathcal{J}\}$ such that there is no descendant of $v$ equal to any of $x_i^{r+1}$, $i\in\mathcal{J}$. By the condition ${\sf IS}_j^{r+1}$, $\mathcal{B}^2_j$ contains $\mathcal{B}^2_*\cong \mathcal{B}^1_*$ and there exists an isomorphism $f:\mathcal{B}^1_*\to \mathcal{B}^2_*$ such that $f(x_i^{r+1})=y_i^{r+1}$, $i\in\mathcal{J}$. 

If $x_{m-2}^{r+1}\in V(\mathcal{B}^1_*)$, then set $f(x_{m-2}^{r+1})=y_{m-2}^{r+1}$. If $x_{m-2}^{r+1}\notin V(\mathcal{B}^1_*)$, then find the closest vertex $v_1$ of $\mathcal{B}^1_*$ to $x_{m-2}^{r+1}$ in $\mathcal{B}^1_{r+1}$. Let $v_2=f(v_1)$. Find a neighbor $v\notin V(\mathcal{B}^2_*)$ of $v_2$ in $\mathcal{B}^2_j$ such that there is no descendant of $v$ equal to $y_i^{r+1}$, $i\in[m-3]$ (the existence of such a neighbor follows from the minimum degree condition). Let $v_2 v\ldots y_{r+1}$ be a path in $\mathcal{B}^2_j$ with length $d_{\mathcal{B}^1_j}(v_1,x_{m-2}^{r+1})$. Clearly, $y_{m-2}^{r+1}$ is the desired vertex.\\

Finally, let $\mathcal{B}^1_{m-2}$ be a tree. 

Let there exist $i\in\mathcal{J}_1$ such that $\mathrm{rd}_i^{r+1}>\mathrm{rd}_j^{r+1}$. Then, $d(x_{m-2}^{r+1},x_i^{r+1})\leq 2^{R-r}$ and $1_i^{r+1}$ holds. Since $d(x_{m-2}^{r+1},x_j^{r+1})\leq 2^{R-r}$, we get that $d(x_i^{r+1},x_j^{r+1})\leq 2^{R+1-r}\leq 2^{R+1-\mathrm{rd}_i^{r+1}}$ and, therefore, $d(x_j^{r+1},[n_0])\leq 2^{m+2-\mathrm{rd}_i^{r+1}}\leq 2^{m+1-\mathrm{rd}_j^{r+1}}$. Then, $j\in\mathcal{J}_3$. Moreover, $d(x_i^{r+1},x_j^{r+1})\leq 2^{m+1-\mathrm{rd}_j^{r+1}}$. Therefore, due to the property ${\sf Q1}$, $x_i^{r+1}\in V_j$ and any shortest path from $x_i^{r+1}$ to $x_j^{r+1}$ leaves $[N_0]$ at the first time at $u_j^{r+1}$.

If there is no $i\in\mathcal{J}_1$ such that $\mathrm{rd}_i^{r+1}>\mathrm{rd}_j^{r+1}$, then $\mathrm{rd}_j^{r+1}$ is the maximum number in $\{\mathrm{rd}_i^{r+1},i\in\mathcal{J}\}$.

Let $\mathcal{B}^1_*$ be obtained from the rooted tree $\mathcal{B}^1_{m-2}$ by removing all vertices $v\notin\{x_i^{r+1},i\in\mathcal{J}\}$ such that there is no descendant of $v$ equal to any of $x_i^{r+1}$, $i\in\mathcal{J}$. By the condition ${\sf IS}_j^{r+1}$, $\mathcal{B}^2_j$ contains $\mathcal{B}^2_*\cong \mathcal{B}^1_*$ and there exists an isomorphism $f:\mathcal{B}^1_*\to \mathcal{B}^2_*$ such that $f(x_i^{r+1})=y_i^{r+1}$ for $i\in\mathcal{J}$ such that $\mathrm{rd}_i^{r+1}\leq\mathrm{rd}_j^{r+1}$.  If $([r]\setminus[j])\cap\mathcal{J}_1=\varnothing$, then this isomorphism preserves all the pebbled vertices but the last one. Let us show that, when there exists $i\in\mathcal{J}_1$ such that $\mathrm{rd}_i^{r+1}>\mathrm{rd}_j^{r+1}$, such an isomorphism also exists. From ${\sf IS}^{r+1}_j$ and $3^{r+1}_j$, we get that we may choose $f$ such that $f(u)=u$ for all $u\in V_j$. Since $x_i^{r+1}\in V_j$ for all $i\in\mathcal{J}$, $i>j$, we get that $f(x_i^{r+1})=y_i^{r+1}$ for all $i\in\mathcal{J}$.

If $x_{m-2}^{r+1}\in V(\mathcal{B}^1_*)$, then set $f(x_{m-2}^{r+1})=y_{m-2}^{r+1}$. If $x_{m-2}^{r+1}\notin V(\mathcal{B}^1_*)$, then find the closest vertex $v_1$ of $\mathcal{B}^1_*$ to $x_{m-2}^{r+1}$ in $\mathcal{B}^1_{r+1}$ and let $v_2=f(v_1)$. Find a neighbor $v\notin V(\mathcal{B}^2_*)$ of $v_2$ in $\mathcal{B}^2_j$ such that there is no descendant of $v$ equal to $y_i^{r+1}$, $i\in[r]$. The desired vertex $y_{m-2}^{r+1}$ is the final vertex of a path $v_2 v\ldots y_{m-2}^{r+1}$ in $\mathcal{B}^2_j$ with length $d_{\mathcal{B}^1_j}(v_1,x_{m-2}^{r+1})$.

\end{enumerate}

\section{Discussions}
\label{conj}

In our paper, we prove $\mathrm{FO}^{m-2}$ convergence law, i.e. we consider only sentences with at most $m-2$ variables. This assumption allows us to significantly simplify the analysis of the local properties of graphs that we study. In particular, for any vertex of $G_{n,m}$ and $a\in\mathbb{N}$, if its $a$-neighborhood
is a tree, then it is $(m-1)$-trivial. The same is true for unicyclic neighborhoods: if $a$-neighborhood of a cycle does not contain any other cycles, then it is $(m-2)$-trivial. 

Lemma~\ref{lem:prob_main} provides that for each $\ell,a$ and $K\in\mathbb{N}$ with high probability there are at least $K$ copies of $C_{\ell}$ having $m$-trivial $a$-neighborhoods (in fact, from the proof it follows that the number of such cycles is asymptotically $c_{l}\ln n$). Without the limitation on the number of variables, we have to consider classes of $\tilde m$-isomorphism of unicyclic graphs for $\tilde m>m-2$. The arguments for tree-neighborhoods still work in this case since: in each $\tilde m$-isomorphism class (for each $\tilde m$ there is a finite number of equivalence classes) there are $\Omega(n)$ tree-neighborhoods with high probability. Moreover, we can still conclude that each graph with a bounded size that contains more than one cycle, with high probability, disappears (as a subgraph) after some moment. However, we can not apply any modification of Lemma~\ref{claim_m-game} since it is very hard to analyze distributions of cardinalities of $\tilde m$-isomorphism classes of unicyclic $a$-neighborhoods. We suggest that some classes contain $\Theta(\ln n)$ graphs, while others contain a finite number of graphs (and the latter cardinalities converge in distribution). The main difficulty of such analysis is that, once a cycle is formed, its $a$-neighborhood may switch between $\tilde m$-isomorphism classes infinitely many times. However, we conjecture that the convergence law holds for all FO sentences.\\

It is also interesting to consider other recursive models in the context of FO limit laws. In particular, our methods can be applied to models that avoid the above mentioned difficulties. For example, one could put a restriction on the maximum degree of the graph (in a similar way to \cite{RW92}) and choose candidates for new attachments only among vertices with degree less than $d$ (for some fixed parameter $d>2m$). In such a model, for every $\tilde m$, cardinalities of classes of $\tilde m$-isomorphism would be finite.

There are recursive random graph models that contain much more dense small subgraphs (for example, graphs with edge-steps (see \cite{ARS19,ARS20})). An approach to prove or disprove logical limit laws for such models should be modified significantly. It is also natural to consider preferential attachment (see, e.g., \cite{MZ20}) models. In preferential attachment random graphs, the probability to form a cycle on vertices with high degrees would increase significantly. In particular, the probability to form a new cycle near an old one would be higher.

\section*{Acknowledgements.}

Maksim Zhukovskii is supported by the Ministry of Science and Higher Education of the Russian Federation in the framework of MegaGrant no 075-15-2019-1926. The part of the study made by Y.A. Malyshkin was funded by RFBR, project number 19-31-60021.


\begin{thebibliography}{99}

\bibitem[ARS19]{ARS19} C. Alves, R. Ribeiro, R. Sanchis, {\it Preferential attachment random graphs with edge-step functions}, Journal of Theoretical Probability, 2019, {\bf 11}.

\bibitem[ARS20]{ARS20} C. Alves, R. Ribeiro, R. Sanchis, {\it Diameter of P.A. random graphs with edge-step functions}. Random Structures and Algorithms, 2020.

\bibitem[BRST]{recursive} B. Bollob\'{a}s, O. Riordan, J. Spencer, G. Tusn\'{a}dy, {\it The degree sequence of a scale-free random graph process}. Random Structures \& Algorithms, 2001, {\bf 18}(3):~279--290.

\bibitem[DL95]{dags} L. Devroye and J. Lu, {\it The strong convergence of maximal degrees in uniform random recursive trees and dags,} Random Structures Algorithms, 1995, {\bf 7}:~1--14. 

\bibitem[E60]{Ehren} A. Ehrenfeucht A. {\it An application of games to the completeness problem for formalized theories}. Warszawa, Fund. Math, 1960, {\bf 49}:~121--149. 

\bibitem[F76]{Fagin} R. Fagin. {\it Probabilities in finite models.} J. Symbolic Logic, 1976, {\bf 41}:~50--58.

\bibitem[GKLT69]{Glebskii} Y. V. Glebskii, D. I. Kogan, M. I. Liogon'kii, V. A. Talanov. {\it Range and degree of realizability of formulas in the restricted predicate calculus}. Cybernetics and Systems Analysis, 1969, {\bf 5}(2):~142--154. (Russian original: Kibernetika, 1969, {\bf 5}(2):~17--27).

\bibitem[HK10]{Haber} S. Haber, M. Krivelevich {\it The logic of random regular graphs}. J. Comb., 2010, {\bf 1}(3-4):~389--440.

\bibitem[HMNT18]{Muller} P.~Heinig, T.~Muller, M.~Noy, A.~Taraz, {\it Logical limit laws for minor-closed classes of graphs}, Journal of Combinatorial Theory, Series B. 2018, {\bf 130}:~158--206.

\bibitem[JKMS]{Magner} S. Janson, G. Kollias, A. Magner, W. Szpankowski, {\it On Symmetry of Uniform and Preferential Attachment Graphs}, The Electronic Journal of Combinatorics, 2014, {\bf 21}(3), P3.32.


\bibitem[JLR00]{Janson}  S. Janson, T. \L uczak, A. Rucinski, {\it Random Graphs}, New York, Wiley, 2000.

\bibitem[KK05]{Kleinberg} R. D. Kleinberg, J. M. Kleinberg. {\it Isomorphism and embedding problems for infinite limits of scale-free graphs}. In Proceedings of the 16th ACM-SIAM Symposium on Discrete Algorithms, 2005, 277--286.


\bibitem[L04]{Libkin} L. Libkin. {\it Elements of finite model theory}. Texts in Theoretical Computer Science. An EATCS Series. Springer-Verlag Berlin Heidelberg. 2004.


\bibitem[MS95]{Mahmoud} H. M. Mahmoud, R. T. Smythe. {\it A survey of recursive trees}. Theory Prob. Math. Statist., 1995, {\bf 51}~:1--27.

\bibitem[MZ20]{MZ20} Y. Malyshkin, M. Zhukovskii, {\it MSO 0-1 law for recursive random trees}, 2020, https://arxiv.org/abs/2007.14768.

\bibitem[M99]{geo} G.L. McColm. {\it First order zero-one laws for random graphs on the circle.} Random Structures and Algorithms, {\bf 14}(3):~239--266, 1999.

\bibitem[M02]{McColm} G.L. McColm. {\it MSO zero-one laws on random labelled acyclic graphs}. Discrete Mathematics, 2002, {\bf 254}:~331--347.

\bibitem[RW92]{RW92} A. Rucinski, N.C. Wormald, {\it Random Graph Processes with Degree Restrictions}.
Combinatorics, Probability and Computing, 1992, {\bf 1}: 169--180.


\bibitem[RZ15]{Survey} A.M.~Raigorodskii, M.E.~Zhukovskii. {\it Random graphs: models and asymptotic characteristics}, Russian Mathematical Surveys, {\bf 70}(1):~33--81, 2015.

\bibitem[SS88]{Shelah} S. Shelah, J.H. Spencer. {\it Zero-one laws for sparse random graphs.} J. Amer. Math. Soc., 1988, {\bf 1}:~97--115.

\bibitem[S91]{Spencer_Ehren} J.H. Spencer. {\it Threshold spectra via the Ehrenfeucht game.} Discrete Applied Math., 1991, {\bf 30}:~235--252.


\bibitem[S01]{Strange} J.H.~Spencer, {\it The Strange Logic of Random Graphs}, Springer Verlag, 2001.

\bibitem[SZ20]{Zhuk_Svesh} N.M. Sveshnikov,  M.E. Zhukovskii, {\it First order zero-one law for uniform random graphs}, Sbornik Mathematics, 2020, {\bf 211}, https://doi.org/10.1070/SM9321.


\bibitem[W93]{Winkler} P. Winkler. {\it Random structures and zero-one laws}. Finite and Infinite Combinatorics in Sets and Logic, N.W. Sauer, R.E. Woodrow and B. Sands, eds., NATO Advanced Science Institute Series, Kluwer Academic Publishers, Dordrecht, 1993, P. 399--420.




\end{thebibliography}
\end{document}